\documentclass[10pt,reqno]{amsart}
\usepackage{latexsym}
\usepackage{amscd}
\usepackage{amssymb}
\usepackage{graphicx}
\usepackage{epstopdf}
\usepackage{amsmath}
\usepackage{textcomp}
\usepackage{cite}

\usepackage{cancel}
\usepackage{tikz}
\usetikzlibrary{decorations.markings}
\tikzstyle{vertex}=[circle, draw, inner sep=0pt, minimum size=6pt]
\newcommand{\vertex}{\node[vertex]}

\usepackage{caption}
\usepackage{subcaption}
\usepackage{amssymb, amscd, amsmath,amsthm,amsfonts,epsfig, enumerate}
\usepackage{epsfig}
\usepackage{fullpage}
\usepackage{listings}
\usepackage{float}
\usepackage{epstopdf}
\usepackage{color}
\def\NZQ{\Bbb}               

\def\ZZ{{\NZQ Z}}

\def\B'c{{\mathcal{B'}}}
\def\U'c{{\mathcal{U'}}}
\def\opn#1#2{\def#1{\operatorname{#2}}} 
\opn\chara{char}
\opn\length{\ell}
\opn\projdim{proj\,dim}
\opn\injdim{inj\,dim}
\opn\ini{in}
\opn\rank{rank}
\opn\depth{depth}
\opn\sdepth{sdepth}
\opn\indmat{indmat}
\opn\cochord{cochord}
\opn\pdim{pdim}
\opn\height{ht}
\opn\embdim{emb\,dim}
\opn\codim{codim}

\opn\Tr{Tr}
\opn\bigrank{big\,rank}
\opn\superheight{superheight}\opn\lcm{lcm}
\opn\trdeg{tr\,deg}%
\opn\reg{reg}
\opn\lreg{lreg}
\opn\set{set}
\opn\supp{Supp}
\opn\shad{Shad}
\opn\brs{brs}
\opn\div{div}
\opn\Div{Div}
\opn\cl{cl}
\opn\Cl{Cl}
\opn\Spec{Spec}
\opn\Supp{Supp}
\opn\supp{supp}
\opn\Sing{Sing}
\opn\Ass{Ass}
\opn\Min{Min}
\opn\size{size}
\opn\bigsize{bigsize}
\opn\lex{lex}
\opn\Ann{Ann}
\opn\Rad{Rad}
\opn\Soc{Soc}
\opn\Ker{Ker}
\opn\Coker{Coker}
\opn\Im{Im}
\opn\Hom{Hom}
\opn\Tor{Tor}
\opn\Ext{Ext}
\opn\End{End}
\opn\Aut{Aut}
\opn\id{id}

\opn\nat{nat}
\opn\GL{GL}
\opn\SL{SL}
\opn\mod{mod}
\opn\ord{ord}
\opn\aff{aff}
\opn\con{conv}
\opn\relint{relint}
\opn\st{st}
\opn\lk{lk}
\opn\cn{cn}
\opn\core{core}
\opn\vol{vol}
\opn\gr{gr}

\def\pot#1#2{#1[\kern-0.28ex[#2]\kern-0.28ex]}

\opn\dirlim{\underrightarrow{\lim}}
\opn\invlim{\underleftarrow{\lim}}
\let\union=\cup

\def\pnt{{\raise0.5mm\hbox{\large\bf.}}}

\def\Implies{\ifmmode\Longrightarrow \else
	\unskip${}\Longrightarrow{}$\ignorespaces\fi}
\def\implies{\ifmmode\Rightarrow \else
	\unskip${}\Rightarrow{}$\ignorespaces\fi}
\def\iff{\ifmmode\Longleftrightarrow \else
	\unskip${}\Longleftrightarrow{}$\ignorespaces\fi}

\let\:=\colon
\newtheorem{Theorem}{Theorem}[section]
\newtheorem{Lemma}[Theorem]{Lemma}
\newtheorem{Corollary}[Theorem]{Corollary}
\newtheorem{Proposition}[Theorem]{Proposition}
\newtheorem{Remark}[Theorem]{Remark}

\newtheorem{Definition}[Theorem]{Definition}

\let\epsilon=\varepsilon
\let\phi=\varphi
\let\kappa=\varkappa
\numberwithin{equation}{section}

\title{some algebraic invariants of the edge ideals of some $q$-fold bristled graphs}
\author[Ayesha Saqib]{Ayesha Saqib}
\address{Ayesha Saqib, School of Natural Sciences, National University of Sciences and Technology Islamabad, Sector H-12, Islamabad Pakistan.}
\email{ayeshasaqib192020@gmail.com}
\author[Muhammad Ishaq]{Muhammad Ishaq}
\address{Muhammad Ishaq, School of Natural Sciences, National University of Sciences and Technology Islamabad, Sector H-12, Islamabad Pakistan.}
\email{ishaq\_maths@yahoo.com}
\begin{document}
	\maketitle
 \begin{abstract}
In this paper, we compute the exact values of regularity of the quotient rings of the edge ideals associated to multi-triangular snake and multi-triangular ouroboros snake graphs. We also compute the exact values of depth, Stanley depth, regularity and projective dimension of the quotient rings of the edge ideals associated to $q$-fold bristled graphs of multi-triangular snake and multi-triangular ouroboros snake graphs. 
\end{abstract}
\section{Introduction}
Let $S:=K[x_{1},x_{2},\dots,x_{n}]$ be a polynomial ring over a field $K$ with standard grading,
that is, $\deg(x_i)=1$, for all $i$. Let $N$ be a finitely generated graded $S$-module. Suppose that $N$ admits the following minimal free resolution:
	$$0\longrightarrow\ \bigoplus_{j \in \mathbb{Z}} S(-j)^{\beta_{t,j}(N)} \xrightarrow{\,\,\psi_{t} \,} \bigoplus_{j \in \mathbb{Z}}  S(-j)^{\beta_{t-1,j}(N)}
	\longrightarrow\ \dots \longrightarrow\  \bigoplus_{j \in \mathbb{Z}} S(-j)^{\beta_{0,j}(N)}\xrightarrow{\,\,\psi_{0} \,} N \longrightarrow\ 0,$$
  the \textit{Castelnuovo-Mumford regularity} (or simply \textit{regularity}) and \textit{projective dimension} of $N$ are defined by $\reg(N)=\max\{j-i:\beta_{i,j}(N)\neq 0\}$ and	$\pdim(N)=\max\{i:\beta_{i,j}(N)\neq 0\},$ respectively.
 The \textit{depth} of $N$ is an algebraic invariant that is defined to be the common length of all maximal $N$-sequences in graded maximal ideal $(x_1,x_2,\dots,x_n)$, we denote the depth of $N$ by $\depth(N)$. Let $N$ be a $\ZZ^n$-graded module over $\ZZ^n$-graded ring $S$. Let $mK[A]$ denotes a $K$-subspace which is generated by all the elements of the form $mv$, where $m$ is a homogeneous element in $N$, $v$ is a monomial in $K[A]$ and $A\subseteq \{x_{1},x_{2},\dots,x_{n}\}$. If $mK[A]$ is a free $K[A]$-module then it is called a Stanley space of dimension $|A|$. A decomposition $\mathcal{D}$ of $K$-vector spaces $N$ as a finite direct sum of Stanley spaces is called a Stanley decomposition of $N$. Let $\mathcal{D} : N=\bigoplus^t_{i=1}m_{i}K[A_{i}],$ the Stanley depth of $\mathcal{D}$ is $\sdepth(\mathcal{D})=\min\{|A_{i}| \colon i=1,2,\dots,t\}$ and the number
	$\sdepth(N)=\max\{\sdepth(\mathcal{D}) \colon \text{$\mathcal{D}$ is a Stanley decomposition of $N$}\},$
	is called the \textit{Stanley depth} of $N$.  For existing literature related to regularity, projective dimension, depth and Stanley depth, we refer the reader to \cite{BR,reg,FS,HJ,PM,UV,VRH}. In 1982, it was conjectured by Stanley \cite{SRP} that for every $\ZZ^n$-graded $\mathit{S}$-module $N$, $\sdepth(N)\geq\depth(N)$. This conjecture was proved in some special cases but in $2016$, Duval et al. in \cite{DG} disproved this conjecture by providing a counterexample. For some recent results related to the said invariants we refer the reader to \cite{IAI,IZ1,IZ2,IZ3,AR1}. 
 
 In this paper, we compute the exact values for regularity of the quotient rings of edge ideals associated to multi-triangular snake and multi-triangular ouroboros snake graphs; see Theorem \ref{Th1} and Theorem \ref{Th3}. We also compute the exact values for depth, Stanley depth, projective dimension and regularity of the quotient rings of edge ideals associated to $q$-fold bristled graphs of multi-triangular snake and multi-triangular ouroboros snake graphs; see Theorem \ref{Th2} and Theorem \ref{Th4}. 
	
 \section{Definitions and Notations}
	Let $G:=(V(G),E(G))$ be a graph with edge set $E(G)$ and vertex set $V(G):=\{x_1,x_2,\dots,x_n\}$. All graphs considered in this paper are simple and undirected. The \textit{edge ideal} $I(G)$ associated to $G$ is a square free monomial ideal of $S$, that is
	$I(G)=(x_{i}x_{j} : \{x_i, x_j\}\in E(G)).$ The minimal set of monomial generators of a monomial ideal $I \subset S$ is denoted by $\mathcal{G}(I)$.  For any monomial $\gamma$, $\supp(\gamma) := \{x_i : x_i |\gamma\}$
	and for a monomial ideal $\mathit{I}$, $\supp(I) := \{x_i : x_i |w,\, \text{for some}\, w \in \mathcal{G}(I)\}.$ The graph whose edge set is empty is called a \textit{null graph}. The graph $G$ is called a \textit{path} if $E(G)=\{\{x_j,x_{j+1}\}:1\leq j\leq n-1\}$, and $G$ is called a \textit{cycle} if $E(G)=\{\{x_j,x_{j+1}\}:1\leq j\leq n-1\}\cup\{\{x_1,x_n\}\}$. A path and a cycle on $n$ vertices are usually denoted by $P_n$ and $C_n$, respectively. A vertex of degree $1$ of a graph is called a \textit{pendant vertex} (\textit{or leaf}). An \textit{internal vertex} is a vertex that is not a pendant vertex.
 
 A \textit{triangular snake} $\mathcal{T}_{n}$ \cite{WT} is a connected graph obtained from a path on vertices $x_1,x_2,\dots,x_{n+1}$ by joining $x_j$ and $x_{j+1}$ to a new vertex $y_j$ for $1\leq j\leq {n}$. In other words, $\mathcal{T}_{n}$ is formed by replacing each edge of $P_{n+1}$ by a triangle $C_3$. A \textit{$p$-triangular snake} $\mathcal{T}_{n}(p)$  (triangular if $p=1$ and multi-triangular if $p\geq2$) consists of $p$ number of triangular snakes that have a common path; see \cite{PRN,SP}. Two vertices $x_1$ and $x_2$ in a graph $G$ are said to be \textit{fused} if $x_1$
	and $x_2$ are replaced by a single new vertex $x$, such that, each edge that was adjacent to either $x_1$ or $x_2$ or both, is adjacent to $x$.
	If we fuse vertices $x_1$ and $x_{n+1}$ in $\mathcal{T}_{n}(p)$, we get a new graph denoted $\mathcal{O}_{n}(p)$ \cite{SMM}, which is known as \textit{$p$-triangular ouroboros snake} (triangular if $p=1$ and multi-triangular if $p\geq2$). We require $n\geq3$ for this construction to be well-defined: if $n=1$, the fusion of $x_1$ and $x_2$ in $\mathcal{T}_1(p)$ would create the edge $\{x_1,x_1\}$, a loop, contradicting the assumption that all graphs are simple. If $n=2$, the fused graph degenerates similarly. Hence throughout this paper $\mathcal{O}_n(p)$ is defined only for $n\geq3$. See Figure \ref{fig:123} for examples of triangular snake, multi-triangular snake and multi-triangular ouroboros snake graphs.
 
 The \textit{corona product} \cite{RF} of two graphs $G$ and $H$ is obtained by taking one copy of $G$ and $\mid V(G)\mid$ copies of $H$ and joining each vertex of the $i^{th}$ copy of $H$ to the $i^{th}$ vertex of $G$, where $1 \leq i \leq \mid V(G)\mid$. For a given graph $G$, its \textit{$q$-fold bristled graph} denoted $\brs_{q}(G)$ is obtained by joining $q$ new vertices to each vertex of $G$; see \cite{KP}. This graph can also be obtained by taking corona product of $G$ with null graph on $q$ vertices. The $q$-fold bristled graph of a given graph is also known as its $q$-thorny graph.
\begin{figure}[H]
	\begin{subfigure}[b]{0.3\textwidth}
		\centering
				\[\begin{tikzpicture}[x=0.5cm, y=0.5cm]
			\vertex[fill] (1) at (0,0) {};
			\vertex[fill] (2) at (2,0) {};
			\vertex[fill] (3) at (4,0) {};
			\vertex[fill] (4) at (6,0) {};
			\vertex[fill] (5) at (1,1) {};
			\vertex[fill] (6) at (3,1) {};
			\vertex[fill] (7) at (5,1) {};
			\path 
			(1) edge (2)
			(3) edge (2)
			(3) edge (4)
			(1) edge (5)
			(6) edge (2)
			(3) edge (7)
			(5) edge (2)
			(3) edge (6)
			(7) edge (4)
			;
			\end{tikzpicture}\]
	\end{subfigure}
	\hfil
	\begin{subfigure}[b]{0.3\textwidth}
	\centering

 \[\begin{tikzpicture}[x=0.5cm, y=0.5cm]
	\vertex[fill] (1) at (0,0) {};
	\vertex[fill] (2) at (2,0) {};
	\vertex[fill] (3) at (4,0) {};
	\vertex[fill] (4) at (6,0) {};
	\vertex[fill] (5) at (1,1) {};
	\vertex[fill] (6) at (3,1) {};
	\vertex[fill] (7) at (5,1) {};
	\vertex[fill] (8) at (1,2) {};
	\vertex[fill] (9) at (3,2) {};
	\vertex[fill] (10) at (5,2) {};
	\vertex[fill] (11) at (1,3){};
	\vertex[fill] (12) at (3,3) {};
	\vertex[fill] (13) at (5,3) {};
	\path
	(1) edge (2)
	(3) edge (2)
	(3) edge (4)
	(1) edge (5)
	(6) edge (2)
	(3) edge (7)
	(5) edge (2)
	(3) edge (6)
	(7) edge (4)
	(1) edge (8)
	(2) edge (8)
	(2) edge (9)
	(3) edge (9)
	(4) edge (10)
	(3) edge (10)
	(1) edge (11)
	(2) edge (11)
	(2) edge (12)
	(3) edge (12)
	(3) edge (13)
	(4) edge (13)
	
	;

	\end{tikzpicture}\]
\end{subfigure}
\hfill
	\begin{subfigure}[b]{0.3\textwidth}
		\centering
		\[\begin{tikzpicture}[x=0.3cm, y=0.3cm]
	\vertex[fill] (1) at (0,0) {};
	\vertex[fill] (2) at (2,0){};
	\vertex[fill] (3) at (2,-2){};
	\vertex[fill] (4) at (0,-2) {};
	\vertex[fill] (5) at (1,2) {};
	\vertex[fill] (6) at (4,-1) {}; 
	\vertex[fill] (7) at (1,-4) {};
	\vertex[fill] (8) at (-2,-1) {};
	\vertex[fill] (9) at (6,-1) {};
	\vertex[fill] (10) at (-4,-1) {};
	\vertex[fill] (11) at (1,4) {};
	\vertex[fill] (12) at (1,-6) {};
	\path
	(1) edge (2)
	(2) edge (3)
	(3) edge (4)
	(1) edge (4)
	(1) edge (5)
	(2) edge (5)
	(2) edge (6)
	(3) edge (6)
	(3) edge (7)
	(4) edge (7)
	(4) edge (8)
	(1) edge (8)
	(1) edge (11)
	(2) edge (11)
	(2) edge (9)
	(3) edge (9)
	(3) edge (12)
	(4) edge (12)
	(1) edge (10)
	(4) edge (10)	
	;
	\end{tikzpicture}\]
	\end{subfigure}
	\caption{From left to right $\mathcal{T}_3$, $\mathcal{T}_3(3)$ and  $\mathcal{O}_4(2)$ }
	\label{fig:123}
\end{figure}
Let $n,p,q\geq1$. The $q$-fold bristled graphs of $p$-triangular snake graph and $p$-triangular ouroboros snake graph are denoted by $\brs_q(\mathcal{T}_n(p))$ and $\brs_q(\mathcal{O}_n(p))$, respectively. See Figure \ref{fig:12} for examples of  $\brs_q(\mathcal{T}_n(p))$ and $\brs_q(\mathcal{O}_n(p))$.
	\begin{Theorem}[{\cite[Theorems 4.7]{reg}}]\label{reg}
		{\em
			Let $ J $ be a monomial ideal and $ y $ a variable of $S$. Then
			
			\begin{itemize}
				\item [(a)]  $ \reg (S/J) = \reg S/(J:y)+1$, if $ \reg (S/J:y) > \reg S/(J,y).$
				\item[(b)] 	  $\reg (S/J)\in \{\reg S/(J,y)+1,\reg S/(J,y)\},$ if $ \reg (S/J:y) = \reg S/(J,y).$
				\item[(c)] 	  $\reg (S/J)= \reg (S/J,y)$ if $ \reg (S/J:y) < \reg S/(J,y).$
			\end{itemize}
		}	
	\end{Theorem}
	\begin{Lemma} [{\cite[Lemma 8]{woodroof}}]\label{circulentt}
		{\em
			Let $ H_1 $  and $ H_2 $ be two disjoint graphs and $ G = H_1\union H_2 $. Then 	
			$	\reg(S/I(G)) = \reg(S/I(H_1)) + \reg(S/I(H_2)).$
		}
	\end{Lemma}
	
	\begin{Lemma}[{\cite{herzog}}](Depth Lemma)\label{depth}
		{\em
If $0\rightarrow {P}_{1} \rightarrow {P}_{2} \rightarrow {P}_{3} \rightarrow 0$ is a short exact sequence of modules over a local ring $S$, or a Noetherian graded ring with $S_{0}$ local, then
			\begin{enumerate}
				\item[(a)] $\depth ({P}_{2}) \geq \min\{\depth({P}_{3}), \depth({P}_{1})\}$.
				\item[(b)] $\depth ({P}_{3}) \geq \min\{\depth({P}_{2}), \depth({P}_{1}) + 1\}$.
				\item[(c)] $\depth ({P}_{1}) \geq \min\{\depth({P}_{3})-1, \depth({P}_{2})\}$.
			\end{enumerate}
		}
	\end{Lemma}
 \begin{figure}[H]
		\begin{subfigure}[b]{0.45\textwidth}
			\centering
			\[\begin{tikzpicture}[x=0.45cm, y=0.45cm]
			\vertex[fill] (1) at (0,0) {};
			\vertex[fill] (2) at (4,0) {};
			\vertex[fill] (3) at (8,0) {};
			\vertex[fill] (4) at (12,0) {};
			\vertex[fill] (5) at (2,2)  {};
			\vertex[fill] (6) at (6,2) {};
			\vertex[fill] (7) at (10,2) {};
			\vertex[fill] (8) at (2,4) {};
			\vertex[fill] (9) at (6,4) {};
			\vertex[fill] (10) at (10,4) {};
			\vertex[fill] (11) at (2,6) {};
			\vertex[fill] (12) at (6,6) {};
			\vertex[fill] (13) at (10,6) {};
			\vertex[fill] (14) at (2,0.5) {};
			\vertex[fill] (15) at (1.5,0.5) {};
			\vertex[fill] (16) at (2.5,0.5){};
			\vertex[fill] (17) at (2,2.5) {};
			\vertex[fill] (18) at (1.5,2.5){}; 	\vertex[fill] (20) at (2.5,2.5) {};
			\vertex[fill] (21) at (2,7.5) {}; 
			\vertex[fill] (22) at (1.5,7.5) {};
			\vertex[fill] (23) at (2.5,7.5) {};
			\vertex[fill] (24) at (6,0.5) {};
			\vertex[fill] (25) at (5.5,0.5) {};
			\vertex[fill] (26) at (6.5,0.5) {};
			\vertex[fill] (27) at (10,0.5) {};
			\vertex[fill] (28) at (9.5,0.5) {};
			\vertex[fill] (29) at (10.5,0.5) {};
			\vertex[fill] (30) at (6,2.5) {};
			\vertex[fill] (31) at (5.5,2.5) {};
			\vertex[fill] (32) at (6.5,2.5) {};
			\vertex[fill] (33) at (10,2.5) {};
			\vertex[fill] (34) at (9.5,2.5) {};
			\vertex[fill] (35) at (10.5,2.5) {};
			\vertex[fill] (36) at (6,7.5) {};
			\vertex[fill] (37) at (5.5,7.5) {};
			\vertex[fill] (38) at (6.5,7.5) {};
			\vertex[fill] (39) at (10,7.5) {};
			\vertex[fill] (40) at (9.5,7.5) {};
			\vertex[fill] (41) at (10.5,7.5) {};
			\vertex[fill] (42) at (0,-1.5) {};
			\vertex[fill] (43) at (-0.5,-1.5) {};
			\vertex[fill] (44) at (0.5,-1.5) {};
			\vertex[fill] (45) at (3.5,-1.5) {};
			\vertex[fill] (46) at (4,-1.5) {};
			\vertex[fill] (47) at (4.5,-1.5) {};
			\vertex[fill] (48) at (7.5,-1.5) {};
			\vertex[fill] (49) at (8,-1.5) {};
			\vertex[fill] (50) at (8.5,-1.5) {};
			\vertex[fill] (51) at (11.5,-1.5) {};
			\vertex[fill] (52) at (12,-1.5) {};
			\vertex[fill] (53) at (12.5,-1.5) {};
			
			\path 
			(1) edge (2)
			(3) edge (2)
			(3) edge (4)
			(1) edge (5)
			(6) edge (2)
			(3) edge (7)
			(5) edge (2)
			(3) edge (6)
			(7) edge (4)
			(1) edge (8)
			(2) edge (8)
			(2) edge (9)
			(3) edge (9)
			(3) edge (10)
			(4) edge (10)
			(1) edge (11)
			(2) edge (11)
			(2) edge (12)
			(3) edge (12)
			(3) edge (13)
			(4) edge (13)
			(5) edge (14)
			(5) edge (15)
			(5) edge (16)
			(8) edge (17)
			(8) edge (18)
			(8) edge (20)
			(11) edge (21)
			(11) edge (23)
			(11) edge (22)
			(6) edge (24)
			(6) edge (25)
			(6) edge (26)
			(7) edge (27)
			(7) edge (28)
			(7) edge (29)
			(9) edge (30)
			(9) edge (31)
			(9) edge (32)
			(10) edge (33)
			(10) edge (34)
			(10) edge (35)
			(12) edge (36)
			(12) edge (37)
			(12) edge (38)
			(13) edge (39)
			(13) edge (40)
			(13) edge (41)
			(1) edge (42)
			(1) edge (43)
			(1) edge (44)
			(2) edge (45)
			(2) edge (46)
			(2) edge (47)
			(3) edge (48)
			(3) edge (49)
			(3) edge (50)
			(4) edge (51)
			(4) edge (52)
			(4) edge (53)
			;
			\end{tikzpicture}\]
		\end{subfigure}
		\hfill
	\begin{subfigure}[b]{0.5\textwidth}
			\centering
			\[\begin{tikzpicture}[x=1.1cm, y=1.1cm]
			\vertex[fill] (1) at (0,0) {};
			\vertex[fill] (2) at (1,0) {};
			\vertex[fill] (3) at (0,1) {};
			\vertex[fill] (4) at (1,1) {};
			\vertex[fill] (5) at (2,0.5) {};
			\vertex[fill] (6) at (3,0.5) {};	\vertex[fill] (7) at (-1,0.5) {};
			\vertex[fill] (8) at (-2,0.5) {};
			\vertex[fill] (9) at (0.5,2) {};
			\vertex[fill] (10) at (0.5,-1) {};	\vertex[fill] (11) at (0.5,-2) {};
			\vertex[fill] (12) at (0.5,3) {};
			\vertex[fill] (13) at (1.3,0.3) {};
			\vertex[fill] (14) at (1.3,0.5) {};
			\vertex[fill] (15) at (1.3,0.7) {};
			\vertex[fill] (16) at (3.7,0.3) {};
			\vertex[fill] (17) at (3.7,0.5) {};
			\vertex[fill] (18) at (3.7,0.7) {};
			\vertex[fill] (19) at (0.3,1.3) {};
			\vertex[fill] (20) at (0.5,1.3) {};
			\vertex[fill] (21) at (0.7,1.3) {};
			\vertex[fill] (22) at (-0.3,0.3) {};
			\vertex[fill] (23) at (-0.3,0.5) {};
			\vertex[fill] (24) at (-0.3,0.7) {};
			\vertex[fill] (25) at (-2.7,0.3) {};
			\vertex[fill] (26) at (-2.7,0.5) {};
			\vertex[fill] (27) at (-2.7,0.7) {};
			\vertex[fill] (28) at (0.3,-0.3) {};
			\vertex[fill] (29) at (0.5,-0.3) {};
			\vertex[fill] (30) at (0.7,-0.3) {};
			\vertex[fill] (31) at (0.3,-2.7) {};
			\vertex[fill] (32) at (0.5,-2.7) {};
			\vertex[fill] (33) at (0.7,-2.7) {};
			\vertex[fill] (34) at (0.3,3.7) {};
			\vertex[fill] (35) at (0.5,3.7) {};
			\vertex[fill] (36) at (0.7,3.7) {};
			\vertex[fill] (37) at (1.65,1.65) {};
			\vertex[fill] (38) at (1.5,1.8) {};
			\vertex[fill] (39) at (1.8,1.5) {};
			\vertex[fill] (40) at (-0.65,1.65) {};
			\vertex[fill] (41) at (-0.5,1.8) {};
			\vertex[fill] (42) at (-0.8,1.5) {};
			\vertex[fill] (43) at (-0.65,-0.65) {};
			\vertex[fill] (44) at (-0.8,-0.5) {};
			\vertex[fill] (45) at (-0.5,-0.8) {};
			\vertex[fill] (46) at (1.5,-0.8) {};
			\vertex[fill] (47) at (1.8,-0.5) {};
			\vertex[fill] (48) at (1.65,-0.65) {};
			\path 
			(1) edge (2)
			(2) edge (4)
			(3) edge (4)
			(1) edge (3)
			(1) edge (10)
			(2) edge (10)
			(1) edge (11)
			(2) edge (11)
			(2) edge (5)
			(4) edge (5)
			(2) edge (6)
			(4) edge (6)
			(3) edge (9)
			(4) edge (12)
			(3) edge (12)
			(4) edge (9)
			(1) edge (7)
			(3) edge (8)
			(1) edge (8)
			(3) edge (7)
			(9) edge (19)
			(9) edge (20)
			(9) edge (21)
			(5) edge (13)
			(5) edge (14)
			(5) edge (15)
			(10) edge (28)
			(10) edge (29)
			(10) edge (30)
			(7) edge (22)
			(7) edge (23)
			(7) edge (24)
			(12) edge (34)
			(12) edge (35)
			(12) edge (36)
			(6) edge (18)
			(6) edge (16)
			(6) edge (17)
			(11) edge (31)
			(11) edge (32)
			(11) edge (33)
			(8) edge (25)
			(8) edge (26)
			(8) edge (27)
			(2) edge (46)
			(2) edge (47)
			(2) edge (48)
			(1) edge (44)
			(1) edge (45)
			(1) edge (43)
			(3) edge (40)
			(3) edge (41)
			(3) edge (42)
			(4) edge (37)
			(4) edge (38)
			(4) edge (39)
			;
			\end{tikzpicture}\]
		\end{subfigure}
	\caption{From left to right $\brs_3(\mathcal{T}_3(3))$  and  $\brs_3(\mathcal{O}_4(2))$ }
		\label{fig:12}
	\end{figure}
	\begin{Lemma} [{\cite[Lemma 2.2]{AR1}}]\label{sdepth} 
		{\em
		For a short exact sequence $0\rightarrow {P}_1\rightarrow {P}_2\rightarrow {P}_3\rightarrow 0$ of $\ZZ^{n}$-graded  ${S}$-modules, we have
			\begin{equation*}
			\sdepth({P}_2)\geq\min\{\sdepth({P}_1),\sdepth({P}_3)\}.
			\end{equation*}
		}
	\end{Lemma}
	\begin{Theorem}[{\cite[Theorems 1.3.3]{BW}}](Auslander–Buchsbaum formula)\label{auss}
		{\em
			If $R$ is a commutative Noetherian local ring and $N$ is a non-zero finitely generated $R$-module of finite projective dimension, then
			\begin{equation*}
			{\pdim}(N)+{\depth}(N)={\depth}(R).
			\end{equation*}
		}
	\end{Theorem}
	\noindent Note that for any monomial ideal $J$ of $S$, $\reg (S/J)=\reg (J)-1$ and $\pdim (S/J)=\pdim (J)+1$.
	\begin{Definition}
		{\em
			Let $u\geq1$. A $u$-star denoted by $\mathcal{S}_{u}$ is a graph on $u+1$ vertices, having one internal vertex of degree ${u}$ and all other vertices of degree $1$.
		}
	\end{Definition}
	\begin{Theorem}\label{star}
		{\em
			 If $u\geq1$, then 
			\begin{itemize}
				\item[(a)] $\depth(K[V(\mathcal{S}_{u})]/I(\mathcal{S}_{u}))=\sdepth(K[V(\mathcal{S}_{u})]/I(\mathcal{S}_{u}))=1,$ {\cite[Theorem 2.6]{AAT}}.
				\item[(b)]$\reg(K[V(\mathcal{S}_{u})]/I(\mathcal{S}_{u}))=1,$ {\cite[Lemma 2.26]{bs}}.
				
			\end{itemize} 
		}
	\end{Theorem}
	The following corollary gives the values of depth, Stanley depth, and regularity for the cyclic module associated to $\brs_q(\mathcal{S}_{u})$.
	\begin{Corollary}\label{bstar}
		{\em
			Let $q,u \geq 1$ and $\mathcal{S}_{u,q}=\brs_q(\mathcal{S}_{u})$. Then 
			\begin{itemize}
				\item[(a)] $\depth(K[V(\mathcal{S}_{u,q})]/I(\mathcal{S}_{u,q}))= \sdepth(K[V(\mathcal{S}_{u,q})]/I(\mathcal{S}_{u,q}))=u+q,$ {\cite[Theorem 2.17]{bs}}.
				\item[(b)]$\reg(K[V(\mathcal{S}_{u,q})]/I(\mathcal{S}_{u,q}))=u,$ {\cite[Theorem 2.34]{bs}}.
				
			\end{itemize} 
		}
	\end{Corollary}

	\begin{Corollary}[{\cite[Corollary 1.3]{AR1}}]\label{c1}
		{\em
			Let $J\subset S$ be a monomial ideal. Then\\ $\depth(S/J)\leq \depth(S/(J:l))$ for all monomials $l\notin J$.
		}
	\end{Corollary}	
	\begin{Proposition}[{\cite[Proposition 2.7]{MC}}]\label{c2} 
		{\em
			Let $J\subset S$ be a monomial ideal. Then\\ $\sdepth(S/J)\leq \sdepth(S/(J:l))$ for all monomials $l\notin J$.
		}
	\end{Proposition}

	\begin{Lemma}[{\cite[Proposition 2.2.21]{VRH}}]\label{ref3}
		{\em
			Let $\mathit{J}_1 \subset \mathit{S^\wedge} = K[x_1 , \dots , x_q],\,\mathit{J}_2 \subset \mathit{S^\diamond}= K[x_{q+1}, \dots , x_n]$ be monomial ideals, with $1 \leq q < n$ and $S=\mathit{S^\wedge}\otimes_{K} \mathit{S^\diamond}$. Then 		
			$$\depth(\mathit{S^\wedge} /\mathit{J}_1\otimes_K\mathit{S^\diamond}/\mathit{J}_2)=\depth_\mathit{S} (\mathit{S}/ (\mathit{J}_1\mathit{S}+ \mathit{J}_2\mathit{S}) )= \depth_{\mathit{S^\wedge}}(\mathit{S^\wedge}/\mathit{J}_1) + \depth_{\mathit{S^\diamond}}(\mathit{S^\diamond}/\mathit{J}_2).$$
		} 
	\end{Lemma} 
	
	\begin{Lemma} [{\cite[Theorem 3.1]{AR1}}]\label{ref5}
		{\em
			Let $\mathit{J}_1 \subset \mathit{S^\wedge} = K[x_1 , \dots , x_q],\,\mathit{J}_2 \subset \mathit{S^\diamond} = K[x_{q+1}, \dots , x_n]$ be monomial ideals, with $1 \leq q < n$ and $S=S^{\wedge}\otimes_K \mathit{S^\diamond}$. Then 		
			$$\sdepth(\mathit{S^\wedge}/\mathit{J}_1\otimes_K\mathit{S^\diamond}/\mathit{J}_2) =\sdepth_\mathit{S} (\mathit{S}/ (\mathit{J}_1\mathit{S}+ \mathit{J}_2\mathit{S}) )\geq \sdepth_{\mathit{S^\wedge}}(\mathit{S^\wedge}/\mathit{J}_1) + \sdepth_{\mathit{S^\diamond}}(\mathit{S^\diamond}/\mathit{J}_2).$$ 
		}
	\end{Lemma}

	\begin{Lemma}\label{ref6}
		{\em
			Let $\mathit{J}\subset \mathit{S}= K[x_1 , \dots , x_n]$ be a monomial ideal and $\mathit{S^\wedge}=\mathit{S}\bigotimes_KK[x_{n+1},\dots,x_{n+s}]$. Then
			\begin{itemize}
				\item[(a)] $\depth(\mathit{S^\wedge}/\mathit{J})=\depth(\mathit{S}/\mathit{J})+s\text{\, and \,}\sdepth(\mathit{S^\wedge}/\mathit{J})=\sdepth(\mathit{S}/\mathit{J})+s,$ {\cite[Lemma 3.6]{herzog}}.
				\item[(b)]$\reg(\mathit{S^\wedge}/\mathit{J})=\reg(\mathit{S}/\mathit{J}) ,$ {\cite[Lemma 3.6]{Morey}}. 
				
			\end{itemize} 
		}
	\end{Lemma}
\begin{Lemma} [{\cite[Lemma 3.1]{AR1}}]\label{IsoJustification}
{\em
Let $G$ be a simple graph with edge ideal $I(G)\subset S=K[V(G)]$ and let $v\in V(G),$ the neighborhood $N(v)=\{u:\{u,v\}\in E(G)\}$ and the closed neighbourhood $N[v]=N(v)\cup\{v\}$. Then:
  $$(I(G):v)=(I(G\setminus N[v]),N(v)) \text{\,\,\,\,\,\,\, and \,\,\,\,\,\,\, } (I(G),v)=(I(G\setminus v),v).$$}
\end{Lemma}
			Let $S_{n,p}:=K[V(\mathcal{T}_n(p))]$ and $S_{n,p,q}:=K[V(\brs_q(\mathcal{T}_n(p)))]$ be polynomial rings whose variables are the vertices of $\mathcal{T}_n(p)$ and $\brs_q(\mathcal{T}_n(p))$, respectively. We denote the edge ideals of $\mathcal{T}_n(p)$ and $\brs_q(\mathcal{T}_n(p))$ by $I_{n,p}$ and $I_{n,p,q}$, respectively. Clearly, $|V(\brs_q(\mathcal{T}_n(p)))|=(1+q)(1+n+np)$ and $|E(\brs_q(\mathcal{T}_n(p)))|=(2p+1)n+(1+n+np)q$. We label the vertices of $\brs_q(\mathcal{T}_n(p))$ in the way as shown in Figure \ref{fig:Tbrs}. If we remove all the pendant vertices from $\brs_q(\mathcal{T}_n(p))$, then we are left with $\mathcal{T}_n(p)$. That is, $\mathcal{T}_n(p)$ is a subgraph induced in $\brs_q(\mathcal{T}_n(p))$ by $V(\mathcal{T}_ n(p))$. Therefore, we use the same labelling for $V(\mathcal{T}_ n(p))$ as we did in  $\brs_q(\mathcal{T}_n(p))$.
   \begin{figure}[H]
			\centering
			\[\begin{tikzpicture}[x=0.7cm, y=0.5cm]
			\vertex[fill] (1) at (0,0) [label=left:$\mathit{x}_1$]{};
			\vertex[fill] (2) at (4,0) [label=above:$\mathit{x}_2$]{};
			\vertex[fill] (3) at (8,0) [label=above:$\mathit{x}_3$]{};
			\vertex[fill] (4) at (12,0) [label=right:$\mathit{x}_4$]{};
			\vertex[fill] (5) at (2,2)  [label=left:$\mathit{y}_{11}$]{};
			\vertex[fill] (6) at (6,2) [label=left:$\mathit{y}_{21}$]{};
			\vertex[fill] (7) at (10,2) [label=left:$\mathit{y}_{31}$]{};
			\vertex[fill] (8) at (2,4) [label=left:$\mathit{y}_{12}$]{};
			\vertex[fill] (9) at (6,4) [label=left:$\mathit{y}_{22}$]{};
			\vertex[fill] (10) at (10,4) [label=left:$\mathit{y}_{32}$]{};
			\vertex[fill] (11) at (2,6) [label=left:$\mathit{y}_{13}$]{};
			\vertex[fill] (12) at (6,6) [label=left:$\mathit{y}_{23}$]{};
			\vertex[fill] (13) at (10,6) [label=left:$\mathit{y}_{33}$] {};
			\vertex[fill] (14) at (2,0.5) [label=below:$\mathit{y}_{112}$]{};
			\vertex[fill] (15) at (1.5,0.5) [label=left:$\mathit{y}_{111}$]{};
			\vertex[fill] (16) at (2.5,0.5) [label=right:$\mathit{y}_{113}$]{};
			\vertex[fill] (17) at (2,2.5) [label=above:$\mathit{y}_{122}$]{};
			\vertex[fill] (18) at (1.5,2.5) [label=left:$\mathit{y}_{121}$]{};
			\vertex[fill] (20) at (2.5,2.5) [label=right:$\mathit{y}_{123}$]{};
			\vertex[fill] (21) at (2,7.5) [label=above:$\mathit{y}_{132}$]{};
			\vertex[fill] (22) at (1.5,7.5) [label=left:$\mathit{y}_{131}$]{};
			\vertex[fill] (23) at (2.5,7.5) [label=right:$\mathit{y}_{133}$]{};
			\vertex[fill] (24) at (6,0.5)[label=below:$\mathit{y}_{212}$] {};
			\vertex[fill] (25) at (5.5,0.5) [label=left:$\mathit{y}_{211}$]{};
			\vertex[fill] (26) at (6.5,0.5) [label=right:$\mathit{y}_{213}$]{};
			\vertex[fill] (27) at (10,0.5) [label=below:$\mathit{y}_{312}$]{};
			\vertex[fill] (28) at (9.5,0.5) [label=left:$\mathit{y}_{311}$]{};
			\vertex[fill] (29) at (10.5,0.5) [label=right:$\mathit{y}_{313}$]{};
			\vertex[fill] (30) at (6,2.5) [label=above:$\mathit{y}_{222}$]{};
			\vertex[fill] (31) at (5.5,2.5) [label=left:$\mathit{y}_{221}$]{};
			\vertex[fill] (32) at (6.5,2.5) [label=right:$\mathit{y}_{223}$]{};
			\vertex[fill] (33) at (10,2.5) [label=above:$\mathit{y}_{322}$]{};
			\vertex[fill] (34) at (9.5,2.5) [label=left:$\mathit{y}_{321}$]{};
			\vertex[fill] (35) at (10.5,2.5) [label=right:$\mathit{y}_{323}$]{};
			\vertex[fill] (36) at (6,7.5) [label=above:$\mathit{y}_{232}$]{};
			\vertex[fill] (37) at (5.5,7.5) [label=left:$\mathit{y}_{231}$]{};
			\vertex[fill] (38) at (6.5,7.5) [label=right:$\mathit{y}_{233}$]{};
			\vertex[fill] (39) at (10,7.5) [label=above:$\mathit{y}_{332}$]{};
			\vertex[fill] (40) at (9.5,7.5) [label=left:$\mathit{y}_{331}$]{};
			\vertex[fill] (41) at (10.5,7.5) [label=right:$\mathit{y}_{333}$]{};
			\vertex[fill] (42) at (0,-1.5) [label=below:$\mathit{x}_{12}$]{};
			\vertex[fill] (43) at (-0.5,-1.5) [label=left:$\mathit{x}_{11}$]{};
			\vertex[fill] (44) at (0.5,-1.5) [label=right:$\mathit{x}_{13}$]{};
			\vertex[fill] (45) at (3.5,-1.5)[label=left:$\mathit{x}_{21}$] {};
			\vertex[fill] (46) at (4,-1.5) [label=below:$\mathit{x}_{22}$]{};
			\vertex[fill] (47) at (4.5,-1.5) [label=right:$\mathit{x}_{23}$]{};
			\vertex[fill] (48) at (7.5,-1.5) [label=left:$\mathit{x}_{31}$]{};
			\vertex[fill] (49) at (8,-1.5) [label=below:$\mathit{x}_{32}$]{};
			\vertex[fill] (50) at (8.5,-1.5) [label=right:$\mathit{x}_{33}$]{};
			\vertex[fill] (51) at (11.5,-1.5) [label=left:$\mathit{x}_{41}$]{};
			\vertex[fill] (52) at (12,-1.5) [label=below:$\mathit{x}_{42}$]{};
			\vertex[fill] (53) at (12.5,-1.5) [label=right:$\mathit{x}_{43}$]{};
			
			\path 
			(1) edge (2)
			(3) edge (2)
			(3) edge (4)
			(1) edge (5)
			(6) edge (2)
			(3) edge (7)
			(5) edge (2)
			(3) edge (6)
			(7) edge (4)
			(1) edge (8)
			(2) edge (8)
			(2) edge (9)
			(3) edge (9)
			(3) edge (10)
			(4) edge (10)
			(1) edge (11)
			(2) edge (11)
			(2) edge (12)
			(3) edge (12)
			(3) edge (13)
			(4) edge (13)
			(5) edge (14)
			(5) edge (15)
			(5) edge (16)
			(8) edge (17)
			(8) edge (18)
			(8) edge (20)
			(11) edge (21)
			(11) edge (23)
			(11) edge (22)
			(6) edge (24)
			(6) edge (25)
			(6) edge (26)
			(7) edge (27)
			(7) edge (28)
			(7) edge (29)
			(9) edge (30)
			(9) edge (31)
			(9) edge (32)
			(10) edge (33)
			(10) edge (34)
			(10) edge (35)
			(12) edge (36)
			(12) edge (37)
			(12) edge (38)
			(13) edge (39)
			(13) edge (40)
			(13) edge (41)
			(1) edge (42)
			(1) edge (43)
			(1) edge (44)
			(2) edge (45)
			(2) edge (46)
			(2) edge (47)
			(3) edge (48)
			(3) edge (49)
			(3) edge (50)
			(4) edge (51)
			(4) edge (52)
			(4) edge (53)
			;
			\end{tikzpicture}\]
			\caption{$\brs_3(\mathcal{T}_{3}(3))$}($3$-fold bristled graph of $\mathcal{T}_3(3)$)
			\label{fig:Tbrs}
		\end{figure}
  
     Let $\mathcal{T}^\star_n(p)$ be the super graph of $\mathcal{T}_n(p)$ that is obtained by joining $p$ number of pendant vertices to vertex $x_{n+1}$ in $\mathcal{T}_n(p)$ and $I^\star_{n,p}:=I(\mathcal{T}^\star_n(p))$. The vertex and edge sets of $\mathcal{T}^\star_n(p)$ are $V(\mathcal{T}^\star_n(p))=V(\mathcal{T}_n(p))
		\bigcup\{\mathit{y}_{(n+1)1},\mathit{y}_{(n+1)2},\dots,\mathit{y}_{(n+1)p}\}$ and $E(\mathcal{T}^\star_n(p))=E(\mathcal{T}_n(p))\bigcup\{\{x_{n+1},y_{(n+1)j}\}:1\leq j\leq p\}.$ Let $\brs_q(\mathcal{T}^\star_n(p))$ be the super graph of $\brs_q(\mathcal{T}_n(p))$ that is obtained by taking one copy of $\brs_q(\mathcal{T}_n(p))$ and $p$ copies of $q$-star and joining the internal vertex of each copy of $q$-star to vertex $x_{n+1}$ in $\brs_q(\mathcal{T}_n(p))$ and  $I^\star_{n,p,q}:=I(\brs_q(\mathcal{T}^\star_n(p))).$ The vertex and edge sets of $\brs_q(\mathcal{T}^\star_n(p))$ are $V(\brs_q(\mathcal{T}^\star_n(p)))=V(\brs_q(\mathcal{T}_n(p)))\bigcup\bigcup\limits_{j=1}^{p}\{\mathit{y}_{(n+1)j},\mathit{y}_{(n+1)j1},\mathit{y}_{(n+1)j2},\dots,\mathit{y}_{(n+1)jq}\}$ and\\
$E(\brs_q(\mathcal{T}^\star_n(p)))=E(\brs_q(\mathcal{T}_n(p)))\bigcup\{\{x_{n+1},y_{(n+1)j}\}:1\leq j\leq p\}\bigcup\{\{y_{(n+1)j},y_{(n+1)jk}\}:1\leq j\leq p \,,\, 1\leq k\leq q\}.$ We label the vertices of $\brs_q(\mathcal{T}^\star_n(p))$ in the way as shown in Figure \ref{fig:subt}.
\begin{figure}[H]
			\centering
			\[\begin{tikzpicture}[x=0.7cm, y=0.5cm]
			\vertex[fill] (1) at (0,0) [label=left:$\mathit{x}_1$]{};
			\vertex[fill] (2) at (4,0) [label=above:$\mathit{x}_2$]{};
			\vertex[fill] (3) at (8,0) [label=right:$\mathit{x}_3$]{};
			\vertex[fill] (5) at (2,2) [label=right:$\mathit{y}_{11}$] {};
			\vertex[fill] (6) at (6,2) [label=right:$\mathit{y}_{21}$]{};
			\vertex[fill] (7) at (10,2) [label=right:$\mathit{y}_{31}$]{};
			\vertex[fill] (8) at (2,4) [label=right:$\mathit{y}_{12}$]{};
			\vertex[fill] (9) at (6,4) [label=right:$\mathit{y}_{22}$]{};
			\vertex[fill] (10) at (10,4) [label=right:$\mathit{y}_{32}$]{};
			\vertex[fill] (11) at (2,6) [label=right:$\mathit{y}_{13}$]{};
			\vertex[fill] (12) at (6,6) [label=right:$\mathit{y}_{23}$]{};
			\vertex[fill] (13) at (10,6) [label=right:$\mathit{y}_{33}$] {};
			\vertex[fill] (14) at (2,0.5) [label=below:$\mathit{y}_{112}$]{};
			\vertex[fill] (15) at (1.5,0.5) [label=left:$\mathit{y}_{111}$]{};
			\vertex[fill] (16) at (2.5,0.5) [label=right:$\mathit{y}_{113}$]{};
			\vertex[fill] (17) at (2,2.5) [label=above:$\mathit{y}_{122}$]{};
			\vertex[fill] (18) at (1.5,2.5) [label=left:$\mathit{y}_{121}$]{};
			\vertex[fill] (20) at (2.5,2.5) [label=right:$\mathit{y}_{123}$]{};
			\vertex[fill] (21) at (2,7.5) [label=above:$\mathit{y}_{132}$]{};
			\vertex[fill] (22) at (1.5,7.5) [label=left:$\mathit{y}_{131}$]{};
			\vertex[fill] (23) at (2.5,7.5)[label=right:$\mathit{y}_{133}$] {};
			\vertex[fill] (24) at (6,0.5) [label=below:$\mathit{y}_{212}$]{};
			\vertex[fill] (25) at (5.5,0.5) [label=left:$\mathit{y}_{211}$]{};
			\vertex[fill] (26) at (6.5,0.5) [label=right:$\mathit{y}_{213}$]{};
			\vertex[fill] (27) at (10,0.5) [label=below:$\mathit{y}_{312}$]{};
			\vertex[fill] (28) at (9.5,0.5) [label=left:$\mathit{y}_{311}$]{};
			\vertex[fill] (29) at (10.5,0.5)[label=right:$\mathit{y}_{313}$] {};
			\vertex[fill] (30) at (6,2.5) [label=above:$\mathit{y}_{222}$]{};
			\vertex[fill] (31) at (5.5,2.5)[label=left:$\mathit{y}_{221}$] {};
			\vertex[fill] (32) at (6.5,2.5) [label=right:$\mathit{y}_{223}$] {};
			\vertex[fill] (33) at (10,2.5) [label=above:$\mathit{y}_{322}$]{};
			\vertex[fill] (34) at (9.5,2.5)[label=left:$\mathit{y}_{321}$] {};
			\vertex[fill] (35) at (10.5,2.5) [label=right:$\mathit{y}_{323}$]{};
			\vertex[fill] (36) at (6,7.5)[label=above:$\mathit{y}_{232}$] {};
			\vertex[fill] (37) at (5.5,7.5) [label=left:$\mathit{y}_{231}$]{};
			\vertex[fill] (38) at (6.5,7.5) [label=right:$\mathit{y}_{233}$]{};
			\vertex[fill] (39) at (10,7.5) [label=above:$\mathit{y}_{332}$]{};
			\vertex[fill] (40) at (9.5,7.5) [label=left:$\mathit{y}_{331}$]{};
			\vertex[fill] (41) at (10.5,7.5)[label=right:$\mathit{y}_{333}$]{};
			\vertex[fill] (42) at (0,-1.5) [label=below:$\mathit{x}_{12}$]{};
			\vertex[fill] (43) at (-0.5,-1.5) [label=left:$\mathit{x}_{11}$]{};
			\vertex[fill] (44) at (0.5,-1.5)[label=right:$\mathit{x}_{13}$] {};
			\vertex[fill] (45) at (3.5,-1.5)[label=left:$\mathit{x}_{21}$] {};
			\vertex[fill] (46) at (4,-1.5) [label=below:$\mathit{x}_{22}$]{};
			\vertex[fill] (47) at (4.5,-1.5)[label=right:$\mathit{x}_{23}$] {};
			\vertex[fill] (48) at (7.5,-1.5) [label=left:$\mathit{x}_{31}$]{};
			\vertex[fill] (49) at (8,-1.5) [label=below:$\mathit{x}_{32}$]{};
			\vertex[fill] (50) at (8.5,-1.5) [label=right:$\mathit{x}_{33}$]{};
			\path 
			(1) edge (2)
			(3) edge (2)
			(1) edge (5)
			(6) edge (2)
			(3) edge (7)
			(5) edge (2)
			(3) edge (6)
			(1) edge (8)
			(2) edge (8)
			(2) edge (9)
			(3) edge (9)
			(3) edge (10)
			(1) edge (11)
			(2) edge (11)
			(2) edge (12)
			(3) edge (12)
			(3) edge (13)
			(5) edge (14)
			(5) edge (15)
			(5) edge (16)
			(8) edge (17)
			(8) edge (18)
			(8) edge (20)
			(11) edge (21)
			(11) edge (23)
			(11) edge (22)
			(6) edge (24)
			(6) edge (25)
			(6) edge (26)
			(7) edge (27)
			(7) edge (28)
			(7) edge (29)
			(9) edge (30)
			(9) edge (31)
			(9) edge (32)
			(10) edge (33)
			(10) edge (34)
			(10) edge (35)
			(12) edge (36)
			(12) edge (37)
			(12) edge (38)
			(13) edge (39)
			(13) edge (40)
			(13) edge (41)
			(1) edge (42)
			(1) edge (43)
			(1) edge (44)
			(2) edge (45)
			(2) edge (46)
			(2) edge (47)
			(3) edge (48)
			(3) edge (49)
			(3) edge (50)
			;
			\end{tikzpicture}\]
			\caption{ ${\brs_3(\mathcal{T}^\star_2(3))}$}
			\label{fig:subt}
		\end{figure}

            Let $Q_{n,p}:=K[V(\mathcal{O}_n(p))]$ and $Q_{n,p,q}:=K[V(\brs_q(\mathcal{O}_n(p)))]$ be polynomial rings whose variables are the vertices of $\mathcal{O}_n(p)$ and $\brs_q(\mathcal{O}_n(p))$, respectively. We denote the edge ideals of $\mathcal{O}_n(p)$ and $\brs_q(\mathcal{O}_n(p))$  by $J_{n,p}$ and $J_{n,p,q}$, respectively. Since we know that $\mathcal{O}_n(p)$ is obtained by fusing the vertices $x_{1}$ and $x_{n+1}$ of $\mathcal{T}_n(p)$. Therefore $V(\mathcal{O}_{n}(p))=V(\mathcal{T}_{n}(p))\setminus\{x_{n+1}\}.$ So the remaining vertices of $\mathcal{T}_{n}(p)$ that are contained in $\mathcal{O}_{n}(p)$ are labelled in the same way as we did in $\mathcal{T}_{n}(p)$. The edge set of $\mathcal{O}_{n}(p)$ is $E(\mathcal{O}_n(p))=E(\mathcal{T}^\star_{n-1}(p))\bigcup\{\{x_{1},x_{n}\}\}\bigcup\{\{x_{1},y_{nj}\}:1\leq j\leq p\}.$ Similarly we have, $V(\brs_q(\mathcal{O}_n(p)))=V(\brs_q(\mathcal{T}_n(p)))\setminus\{x_{n+1},x_{(n+1)1},x_{(n+1)2},\dots,x_{(n+1)q}\}$ and $E(\brs_q(\mathcal{O}_n(p)))=E(\brs_q(\mathcal{T}^\star_{n-1}(p)))\bigcup\{\{x_{1},x_{n}\}\}\bigcup\{\{x_{1},y_{nj}\}:1\leq j\leq p\}.$ Clearly, $|V(\brs_q(\mathcal{O}_n(p)))|=(q+1)(p+1)n$ and $|E(\brs_q(\mathcal{O}_n(p)))|=n(q(p+1)+2p+1)$. 
            \begin{figure}[H]
			\centering
			\[\begin{tikzpicture}[x=0.7cm, y=0.5cm]
			\vertex[fill] (1) at (0,0) [label=left:$\mathit{x}_1$]{};
			\vertex[fill] (2) at (4,0) [label=above:$\mathit{x}_2$]{};
			\vertex[fill] (3) at (8,0) [label=right:$\mathit{x}_3$]{};
			\vertex[fill] (5) at (2,2) [label=right:$\mathit{y}_{11}$] {};
			\vertex[fill] (6) at (6,2) [label=right:$\mathit{y}_{21}$]{};
			\vertex[fill] (7) at (10,2) [label=right:$\mathit{y}_{31}$]{};
			\vertex[fill] (8) at (2,4) [label=right:$\mathit{y}_{12}$]{};
			\vertex[fill] (9) at (6,4) [label=right:$\mathit{y}_{22}$]{};
			\vertex[fill] (10) at (10,4) [label=right:$\mathit{y}_{32}$]{};
			\vertex[fill] (11) at (2,6) [label=right:$\mathit{y}_{13}$]{};
			\vertex[fill] (12) at (6,6) [label=right:$\mathit{y}_{23}$]{};
			\vertex[fill] (13) at (10,6) [label=right:$\mathit{y}_{33}$] {};
			\vertex[fill] (14) at (2,0.5) [label=below:$\mathit{y}_{112}$]{};
			\vertex[fill] (15) at (1.5,0.5) [label=left:$\mathit{y}_{111}$]{};
			\vertex[fill] (16) at (2.5,0.5) [label=right:$\mathit{y}_{113}$]{};
			\vertex[fill] (17) at (2,2.5) [label=above:$\mathit{y}_{122}$]{};
			\vertex[fill] (18) at (1.5,2.5) [label=left:$\mathit{y}_{121}$]{};
			\vertex[fill] (20) at (2.5,2.5) [label=right:$\mathit{y}_{123}$]{};
			\vertex[fill] (21) at (2,7.5) [label=above:$\mathit{y}_{132}$]{};
			\vertex[fill] (22) at (1.5,7.5) [label=left:$\mathit{y}_{131}$]{};
			\vertex[fill] (23) at (2.5,7.5)[label=right:$\mathit{y}_{133}$] {};
			\vertex[fill] (24) at (6,0.5) [label=below:$\mathit{y}_{212}$]{};
			\vertex[fill] (25) at (5.5,0.5) [label=left:$\mathit{y}_{211}$]{};
			\vertex[fill] (26) at (6.5,0.5) [label=right:$\mathit{y}_{213}$]{};
			\vertex[fill] (27) at (10,0.5) [label=below:$\mathit{y}_{312}$]{};
			\vertex[fill] (28) at (9.5,0.5) [label=left:$\mathit{y}_{311}$]{};
			\vertex[fill] (29) at (10.5,0.5)[label=right:$\mathit{y}_{313}$] {};
			\vertex[fill] (30) at (6,2.5) [label=above:$\mathit{y}_{222}$]{};
			\vertex[fill] (31) at (5.5,2.5)[label=left:$\mathit{y}_{221}$] {};
			\vertex[fill] (32) at (6.5,2.5) [label=right:$\mathit{y}_{223}$] {};
			\vertex[fill] (33) at (10,2.5) [label=above:$\mathit{y}_{322}$]{};
			\vertex[fill] (34) at (9.5,2.5)[label=left:$\mathit{y}_{321}$] {};
			\vertex[fill] (35) at (10.5,2.5) [label=right:$\mathit{y}_{323}$]{};
			\vertex[fill] (36) at (6,7.5)[label=above:$\mathit{y}_{232}$] {};
			\vertex[fill] (37) at (5.5,7.5) [label=left:$\mathit{y}_{231}$]{};
			\vertex[fill] (38) at (6.5,7.5) [label=right:$\mathit{y}_{233}$]{};
			\vertex[fill] (39) at (10,7.5) [label=above:$\mathit{y}_{332}$]{};
			\vertex[fill] (40) at (9.5,7.5) [label=left:$\mathit{y}_{331}$]{};
			\vertex[fill] (41) at (10.5,7.5)[label=right:$\mathit{y}_{333}$]{};
			\vertex[fill] (42) at (0,-1.5) [label=below:$\mathit{x}_{12}$]{};
			\vertex[fill] (43) at (-0.5,-1.5) [label=left:$\mathit{x}_{11}$]{};
			\vertex[fill] (44) at (0.5,-1.5)[label=right:$\mathit{x}_{13}$] {};
			\vertex[fill] (45) at (3.5,-1.5)[label=left:$\mathit{x}_{21}$] {};
			\vertex[fill] (46) at (4,-1.5) [label=below:$\mathit{x}_{22}$]{};
			\vertex[fill] (47) at (4.5,-1.5)[label=right:$\mathit{x}_{23}$] {};
			\vertex[fill] (48) at (7.5,-1.5) [label=left:$\mathit{x}_{31}$]{};
			\vertex[fill] (49) at (8,-1.5) [label=below:$\mathit{x}_{32}$]{};
			\vertex[fill] (50) at (8.5,-1.5) [label=right:$\mathit{x}_{33}$]{};
			\vertex[fill] (51) at (-2,2) [label=right:$\mathit{y}_{41}$]{};
			\vertex[fill] (52) at (-2,4) [label=right:$\mathit{y}_{42}$]{};
			\vertex[fill] (53) at (-2,6) [label=right:$\mathit{y}_{43}$]{};
			\vertex[fill] (54) at (-2,0.5) [label=below:$\mathit{y}_{412}$]{};
			\vertex[fill] (55) at (-1.5,0.5) [label=right:$\mathit{y}_{413}$]{};
			\vertex[fill] (56) at (-2.5,0.5) [label=left:$\mathit{y}_{411}$]{};
			\vertex[fill] (57) at (-2,2.5) [label=above:$\mathit{y}_{422}$]{};
			\vertex[fill] (58) at (-1.5,2.5) [label=right:$\mathit{y}_{423}$]{};
			\vertex[fill] (59) at (-2.5,2.5) [label=left:$\mathit{y}_{421}$]{};
			\vertex[fill] (60) at (-2,7.5) [label=above:$\mathit{y}_{432}$]{};
			\vertex[fill] (61) at (-1.5,7.5) [label=right:$\mathit{y}_{433}$]{};
			\vertex[fill] (62) at (-2.5,7.5) [label=left:$\mathit{y}_{431}$]{};

			\path 
			(1) edge (2)
			(3) edge (2)
			(1) edge (5)
			(6) edge (2)
			(3) edge (7)
			(5) edge (2)
			(3) edge (6)
			(1) edge (8)
			(2) edge (8)
			(2) edge (9)
			(3) edge (9)
			(3) edge (10)
			(1) edge (11)
			(2) edge (11)
			(2) edge (12)
			(3) edge (12)
			(3) edge (13)
			(5) edge (14)
			(5) edge (15)
			(5) edge (16)
			(8) edge (17)
			(8) edge (18)
			(8) edge (20)
			(11) edge (21)
			(11) edge (23)
			(11) edge (22)
			(6) edge (24)
			(6) edge (25)
			(6) edge (26)
			(7) edge (27)
			(7) edge (28)
			(7) edge (29)
			(9) edge (30)
			(9) edge (31)
			(9) edge (32)
			(10) edge (33)
			(10) edge (34)
			(10) edge (35)
			(12) edge (36)
			(12) edge (37)
			(12) edge (38)
			(13) edge (39)
			(13) edge (40)
			(13) edge (41)
			(1) edge (42)
			(1) edge (43)
			(1) edge (44)
			(2) edge (45)
			(2) edge (46)
			(2) edge (47)
			(3) edge (48)
			(3) edge (49)
			(3) edge (50)
			(51) edge (55)
			(51) edge (56)
			(51) edge (54)
			(52) edge (57)
			(52) edge (58)
			(52) edge (59)
			(53) edge (60)
			(53) edge (61)
			(53) edge (62)
			(1) edge (51)
			(1) edge (52)
			(1) edge (53)

			;
			\end{tikzpicture}\]
			\caption{ ${\brs_3(\mathcal{T}^{\star\star}_{2}(3))}$}
			\label{fig:subt2}
		\end{figure}
Let $\mathcal{T}^{\star\star}_{n}(p)$ be the super graph of $\mathcal{T}^{\star}_{n}(p)$ that is obtained by joining $p$ number of pendant vertices to vertex $x_1$ in $\mathcal{T}^{\star}_{n}(p)$ and $J^\star_{n,p}:=I(\mathcal{T}^{\star\star}_{n}(p)).$ The vertex and edge sets of $\mathcal{T}^{\star\star}_{n}(p)$ are $V(\mathcal{T}^{\star\star}_{n}(p))=V(\mathcal{T}^{\star}_{n}(p))\bigcup\{\mathit{y}_{(n+2)1},\mathit{y}_{(n+2)2},\dots,\mathit{y}_{(n+2)p}\}$ and $E(\mathcal{T}^{\star\star}_{n}(p))=E(\mathcal{T}^{\star}_{n}(p))\bigcup\{\{x_{1},y_{(n+2)j}\}:1\leq j\leq p\}.$ Let $\brs_q(\mathcal{T}^{\star\star}_{n}(p))$ be the super graph of $\brs_q(\mathcal{T}^{\star}_{n}(p))$ that is obtained by taking one copy of $\brs_q(\mathcal{T}^{\star}_{n}(p))$ and $p$ copies of $q$-star and joining internal vertex of each copy of $q$-star to vertex $x_1$ in $\brs_q(\mathcal{T}^{\star}_{n}(p))$ and $J^\star_{n,p,q}:=I(\brs_q(\mathcal{T}^{\star\star}_{n}(p))$). The vertex and edge sets of $\brs_q(\mathcal{T}^{\star\star}_{n}(p))$ are $V(\brs_q(\mathcal{T}^{\star\star}_{n}(p)))=V(\brs_q(\mathcal{T}^{\star}_{n}(p)))
\bigcup\bigcup\limits_{j=1}^{p}\{\mathit{y}_{(n+2)j},\mathit{y}_{(n+2)j1},\mathit{y}_{(n+2)j2},\dots,\mathit{y}_{(n+2)jq}\}$ and \\$E(\brs_q(\mathcal{T}^{\star\star}_{n}(p)))=E(\brs_q(\mathcal{T}^{\star}_{n}(p)))\bigcup\{\{x_{1},y_{(n+2)j}\}:1\leq j\leq p\}\bigcup\{\{y_{(n+2)j},y_{(n+2)jk}\}:1\leq j\leq p \,,\, 1\leq k\leq q\}.$ We label the vertices of $\brs_q(\mathcal{T}^{\star\star}_{n}(p))$ in the way as shown in Figure \ref{fig:subt2}.
We consider the polynomial rings $S^\star_{n,p}:=K[V(\mathcal{T}^\star_n(p))]$, $S^\star_{n,p,q}:=K[V(\brs_q(\mathcal{T}^\star_n(p)))]$, $Q^\star_{n,p}:=K[V(\mathcal{T}^{\star\star}_n(p))]$ and $Q^\star_{n,p,q}:=K[V(\brs_q(\mathcal{T}^{\star\star}_n(p)))]$.\\

  Now we consider some subsets of $V(\brs_q(\mathcal{T}^{\star\star}_{n}(p)))$ that will be used frequently in this paper, $\mathit{A}_{i}:=\{\mathit{x}_{i1},\mathit{x}_{i2},\dots,\mathit{x}_{iq}\}$, $\mathit{B}_{j}:=\{\mathit{y}_{j1},\mathit{y}_{j2},\dots,\mathit{y}_{jp}\}$ and $\mathit{C}_{j}:=\bigcup\limits_{k=1}^{p}\{\mathit{y}_{jk1},\mathit{y}_{jk2},\dots,\mathit{y}_{jkq}\}$, for all $i$ and $j$. We show that the values of depth and Stanley depth are equal, which proves the Stanley\textquoteright s inequality for considered modules. \\

	\begin{Remark}
		{\em
			Let $\mathit{J}$ be a square free monomial ideal of $\mathit{S}$ minimally generated by monomials of degree at most $2$. We associate a graph $\mathit{G_{J}}$ to the ideal $\mathit{J}$ with $\mathit{V(G_{J})}=\supp(\mathit{J})$ and $\mathit{E(G_{J})}=\{\{x_i,x_j\} : x_ix_j \in \mathcal{G}(J)\}$. Let $x_r \in \mathit{S}$ be a variable of the polynomial ring $\mathit{S}$ such that $x_r \notin \mathit{J}$. Then $(\mathit{J}:x_r)$ and $(\mathit{J},x_r)$  are monomial ideals of $\mathit{S}$ such that $G_{(\mathit{J},x_r)}$ and $G_{(\mathit{J}:x_r)}$ are subgraphs of $G_{\mathit{J}}$. See Figures \ref{fig:G1} and \ref{fig:G2} for the examples of $G_{(\mathit{I_{n,p,q}}:x_{n+1})}$ and $G_{(\mathit{I_{n,p,q}},x_{n+1})}$, respectively. And see Figures \ref{fig:G3} and \ref{fig:G4} for the examples of $G_{(\mathit{J_{n,p,q}},x_{n+1})}$ and $G_{(\mathit{J_{n,p,q}}:x_{n+1})}$, respectively. For example, taking $n=3$, $p=3$, $q=3$, we have the following isomorphisms:
			$$\mathit{S}_{3,3,3}/(\mathit{I}_{3,3,3},\mathit{x}_{4})\cong\mathit{S}_{3,3,3}/I(G_{(\mathit{I}_{3,3,3},x_4)})\cong \mathit{S}^\star_{2,3,3}/\mathit{I}^\star_{2,3,3}\bigotimes_KK[\mathit{A}_{4}],$$
			$$\mathit{S}_{3,3,3}/(\mathit{I}_{3,3,3}:\mathit{x}_{4})\cong\mathit{S}_{3,3,3}/I(G_{(\mathit{I}_{3,3,3}:x_4)})\cong \mathit{S}^\star_{1,3,3}/\mathit{I}^\star_{1,3,3}\bigotimes_KK[\{\mathit{x}_4\}\cup\mathit{A}_{3}\cup\mathit{C}_{3}],$$
			$$\mathit{Q}_{3,3,3}/(\mathit{J}_{3,3,3},\mathit{x}_{4})\cong\mathit{Q}_{3,3,3}/J(G_{(\mathit{J}_{3,3,3},x_4)})\cong \mathit{Q}^\star_{2,3,3}/\mathit{J}^\star_{2,3,3}\bigotimes_KK[\mathit{A}_{4}],$$
			$$\mathit{Q}_{3,3,3}/(\mathit{J}_{3,3,3}:\mathit{x}_{4})\cong K[V(\mathcal{S}_{4,3})]/I(\mathcal{S}_{4,3})\bigotimes_KK[\{\mathit{x}_4\}\cup\mathit{A}_{1}\cup\mathit{A}_{3}\cup\mathit{C}_{3}\cup\mathit{C}_{4}].$$
		}
	\end{Remark}


	\section{Depth, Stanley Depth, regularity and projective dimension of cyclic modules associated to $\mathcal{T}_n$, $\mathcal{T}_n(p)$, $\brs_q(\mathcal{T}_n)$ and $\brs_q(\mathcal{T}_n(p))$}\label{sec3}
	In this section we compute the exact value of regularity for the cyclic module $\mathit{S}_{n,p}/\mathit{I}_{n,p}$. For this purpose we first compute the exact value of regularity for the cyclic module $\mathit{S}^\star_{n,p}/\mathit{I}^\star_{n,p}$. Shahid et al. in \cite{SMM} gave the values and tight bounds of depth and Stanley depth for these modules. The values and bounds of projective dimension for these modules can be found by using Theorem \ref{auss}. Further we compute the exact values of depth, Stanley depth, regularity and projective dimension for the cyclic module $\mathit{S}_{n,p,q}/\mathit{I}_{n,p,q}$. For this purpose we first compute the exact values of all mentioned invariants for the cyclic module $\mathit{S}^\star_{n,p,q}/\mathit{I}^\star_{n,p,q}$. \begin{figure}[H]
		\begin{minipage}{0.5\textwidth}
			\centering
			\[\begin{tikzpicture}[x=0.45cm, y=0.45cm]
			\vertex[fill] (1) at (0,0) {};
			\vertex[fill] (2) at (4,0) {};
			\vertex[fill] (3) at (8,0) {};
			\vertex[fill] (4) at (12,0) {};
			\vertex[fill] (5) at (2,2)  {};
			\vertex[fill] (6) at (6,2) {};
			\vertex[fill] (7) at (10,2) {};
			\vertex[fill] (8) at (2,4) {};
			\vertex[fill] (9) at (6,4) {};
			\vertex[fill] (10) at (10,4) {};
			\vertex[fill] (11) at (2,6) {};
			\vertex[fill] (12) at (6,6) {};
			\vertex[fill] (13) at (10,6)  {};
			\vertex[fill] (14) at (2,0.5) {};
			\vertex[fill] (15) at (1.5,0.5) {};
			\vertex[fill] (16) at (2.5,0.5) {};
			\vertex[fill] (17) at (2,2.5) {};
			\vertex[fill] (18) at (1.5,2.5) {};
			\vertex[fill] (20) at (2.5,2.5) {};
			\vertex[fill] (21) at (2,7.5) {};
			\vertex[fill] (22) at (1.5,7.5) {};
			\vertex[fill] (23) at (2.5,7.5) {};
			\vertex[fill] (24) at (6,0.5){};
			\vertex[fill] (25) at (5.5,0.5) {};
			\vertex[fill] (26) at (6.5,0.5) {};
			\vertex[fill] (27) at (10,0.5) {};
			\vertex[fill] (28) at (9.5,0.5) {};
			\vertex[fill] (29) at (10.5,0.5) {};
			\vertex[fill] (30) at (6,2.5) {};
			\vertex[fill] (31) at (5.5,2.5) {};
			\vertex[fill] (32) at (6.5,2.5) {};
			\vertex[fill] (33) at (10,2.5) {};
			\vertex[fill] (34) at (9.5,2.5) {};
			\vertex[fill] (35) at (10.5,2.5) {};
			\vertex[fill] (36) at (6,7.5) {};
			\vertex[fill] (37) at (5.5,7.5) {};
			\vertex[fill] (38) at (6.5,7.5) {};
			\vertex[fill] (39) at (10,7.5) {};
			\vertex[fill] (40) at (9.5,7.5) {};
			\vertex[fill] (41) at (10.5,7.5) {};
			\vertex[fill] (42) at (0,-1.5) {};
			\vertex[fill] (43) at (-0.5,-1.5) {};
			\vertex[fill] (44) at (0.5,-1.5) {};
			\vertex[fill] (45) at (3.5,-1.5) {};
			\vertex[fill] (46) at (4,-1.5) {};
			\vertex[fill] (47) at (4.5,-1.5) {};
			\vertex[fill] (48) at (7.5,-1.5) {};
			\vertex[fill] (49) at (8,-1.5) {};
			\vertex[fill] (50) at (8.5,-1.5) {};
			
			\path 
			(1) edge (2)
			(3) edge (2)
			(1) edge (5)
			(6) edge (2)
			(3) edge (7)
			(5) edge (2)
			(3) edge (6)
			(1) edge (8)
			(2) edge (8)
			(2) edge (9)
			(3) edge (9)
			(3) edge (10)
			(1) edge (11)
			(2) edge (11)
			(2) edge (12)
			(3) edge (12)
			(3) edge (13)
			(5) edge (14)
			(5) edge (15)
			(5) edge (16)
			(8) edge (17)
			(8) edge (18)
			(8) edge (20)
			(11) edge (21)
			(11) edge (23)
			(11) edge (22)
			(6) edge (24)
			(6) edge (25)
			(6) edge (26)
			(7) edge (27)
			(7) edge (28)
			(7) edge (29)
			(9) edge (30)
			(9) edge (31)
			(9) edge (32)
			(10) edge (33)
			(10) edge (34)
			(10) edge (35)
			(12) edge (36)
			(12) edge (37)
			(12) edge (38)
			(13) edge (39)
			(13) edge (40)
			(13) edge (41)
			(1) edge (42)
			(1) edge (43)
			(1) edge (44)
			(2) edge (45)
			(2) edge (46)
			(2) edge (47)
			(3) edge (48)
			(3) edge (49)
			(3) edge (50)
			
			;
			\end{tikzpicture}\]
			\caption{$G_{({I_{3,3,3}},x_{4})}$}
			\label{fig:G1}
		\end{minipage}\hfill
		\begin{minipage}{0.5\textwidth}
			\centering
			\[\begin{tikzpicture}[x=0.45cm, y=0.45cm]
			\vertex[fill] (1) at (0,0) {};
			\vertex[fill] (2) at (4,0) {};
			\vertex[fill] (3) at (8,0) {};
			
			\vertex[fill] (5) at (2,2)  {};
			\vertex[fill] (6) at (6,2) {};
			\vertex[fill] (7) at (10,2) {};
			\vertex[fill] (8) at (2,4) {};
			\vertex[fill] (9) at (6,4) {};
			\vertex[fill] (10) at (10,4) {};
			\vertex[fill] (11) at (2,6) {};
			\vertex[fill] (12) at (6,6) {};
			\vertex[fill] (13) at (10,6)  {};
			\vertex[fill] (14) at (2,0.5) {};
			\vertex[fill] (15) at (1.5,0.5) {};
			\vertex[fill] (16) at (2.5,0.5) {};
			\vertex[fill] (17) at (2,2.5) {};
			\vertex[fill] (18) at (1.5,2.5) {};
			\vertex[fill] (20) at (2.5,2.5) {};
			\vertex[fill] (21) at (2,7.5) {};
			\vertex[fill] (22) at (1.5,7.5) {};
			\vertex[fill] (23) at (2.5,7.5) {};
			\vertex[fill] (24) at (6,0.5){};
			\vertex[fill] (25) at (5.5,0.5) {};
			\vertex[fill] (26) at (6.5,0.5) {};
			\vertex[fill] (30) at (6,2.5) {};
			\vertex[fill] (31) at (5.5,2.5){};
			\vertex[fill] (32) at (6.5,2.5) {};
			\vertex[fill] (36) at (6,7.5) {};
			\vertex[fill] (37) at (5.5,7.5) {};
			\vertex[fill] (38) at (6.5,7.5) {};
			\vertex[fill] (42) at (0,-1.5) {};
			\vertex[fill] (43) at (-0.5,-1.5) {};
			\vertex[fill] (44) at (0.5,-1.5) {};
			\vertex[fill] (45) at (3.5,-1.5){};
			\vertex[fill] (46) at (4,-1.5) {};
			\vertex[fill] (47) at (4.5,-1.5) {};
			\vertex[fill] (51) at (11.5,-1.5) {};
			\vertex[fill] (52) at (12,-1.5) {};
			\vertex[fill] (53) at (12.5,-1.5) {};
			
			\path 
			(1) edge (2)
			(1) edge (5)
			(6) edge (2)
			(5) edge (2)
			(1) edge (8)
			(2) edge (8)
			(2) edge (9)
			(1) edge (11)
			(2) edge (11)
			(2) edge (12)
			(5) edge (14)
			(5) edge (15)
			(5) edge (16)
			(8) edge (17)
			(8) edge (18)
			(8) edge (20)
			(11) edge (21)
			(11) edge (23)
			(11) edge (22)
			(6) edge (24)
			(6) edge (25)
			(6) edge (26)
			(9) edge (30)
			(9) edge (31)
			(9) edge (32)
			(12) edge (36)
			(12) edge (37)
			(12) edge (38)
			(1) edge (42)
			(1) edge (43)
			(1) edge (44)
			(2) edge (45)
			(2) edge (46)
			(2) edge (47)

			;
			\end{tikzpicture}\]
			\caption{$G_{(I_{3,3,3}:x_{4})}$}
			\label{fig:G2}
		\end{minipage}
	\end{figure}
	\begin{figure}[H]
		\begin{minipage}{0.5\textwidth}
			\centering
			\[\begin{tikzpicture}[x=1cm, y=1cm]
			\vertex[fill] (1) at (0,0) {};
			\vertex[fill] (2) at (1,0) {};
			\vertex[fill] (3) at (0,1) {};
			\vertex[fill] (4) at (1,1) {};
			\vertex[fill] (5) at (2,0.5) {};
			\vertex[fill] (6) at (3,0.5) {};
			\vertex[fill] (7) at (-1,0.5) {};
			\vertex[fill] (8) at (-2,0.5) {};
			\vertex[fill] (9) at (0.5,2) {};
			\vertex[fill] (10) at (0.5,-1) {};
			\vertex[fill] (11) at (0.5,-2) {};
			\vertex[fill] (12) at (0.5,3) {};
			\vertex[fill] (13) at (1.3,0.3) {};
			\vertex[fill] (14) at (1.3,0.5) {};
			\vertex[fill] (15) at (1.3,0.7) {};
			\vertex[fill] (16) at (3.7,0.3) {};
			\vertex[fill] (17) at (3.7,0.5) {};
			\vertex[fill] (18) at (3.7,0.7) {};
			\vertex[fill] (19) at (0.3,1.3) {};
			\vertex[fill] (20) at (0.5,1.3) {};
			\vertex[fill] (21) at (0.7,1.3) {};
			\vertex[fill] (22) at (-0.3,0.3) {};
			\vertex[fill] (23) at (-0.3,0.5) {};
			\vertex[fill] (24) at (-0.3,0.7) {};
			\vertex[fill] (25) at (-2.7,0.3) {};
			\vertex[fill] (26) at (-2.7,0.5) {};
			\vertex[fill] (27) at (-2.7,0.7) {};
			\vertex[fill] (28) at (0.3,-0.3) {};
			\vertex[fill] (29) at (0.5,-0.3) {};
			\vertex[fill] (30) at (0.7,-0.3) {};
			\vertex[fill] (31) at (0.3,-2.7) {};
			\vertex[fill] (32) at (0.5,-2.7) {};
			\vertex[fill] (33) at (0.7,-2.7) {};
			\vertex[fill] (34) at (0.3,3.7) {};
			\vertex[fill] (35) at (0.5,3.7) {};
			\vertex[fill] (36) at (0.7,3.7) {};
			\vertex[fill] (37) at (1.65,1.65) {};
			\vertex[fill] (38) at (1.5,1.8) {};
			\vertex[fill] (39) at (1.8,1.5) {};
			\vertex[fill] (40) at (-0.65,1.65) {};
			\vertex[fill] (41) at (-0.5,1.8) {};
			\vertex[fill] (42) at (-0.8,1.5) {};
			
			\vertex[fill] (46) at (1.5,-0.8) {};
			\vertex[fill] (47) at (1.8,-0.5) {};
			\vertex[fill] (48) at (1.65,-0.65) {};
			\path 
			
			(2) edge (4)
			(3) edge (4)

			(2) edge (10)
			
			(2) edge (11)
			(2) edge (5)
			(4) edge (5)
			(2) edge (6)
			(4) edge (6)
			(3) edge (9)
			(4) edge (12)
			(3) edge (12)
			(4) edge (9)
			
			(3) edge (8)
			
			(3) edge (7)
			(9) edge (19)
			(9) edge (20)
			(9) edge (21)
			(5) edge (13)
			(5) edge (14)
			(5) edge (15)
			(10) edge (28)
			(10) edge (29)
			(10) edge (30)
			(7) edge (22)
			(7) edge (23)
			(7) edge (24)
			(12) edge (34)
			(12) edge (35)
			(12) edge (36)
			(6) edge (18)
			(6) edge (16)
			(6) edge (17)
			(11) edge (31)
			(11) edge (32)
			(11) edge (33)
			(8) edge (25)
			(8) edge (26)
			(8) edge (27)
			(2) edge (46)
			(2) edge (47)
			(2) edge (48)
			(3) edge (40)
			(3) edge (41)
			(3) edge (42)
			(4) edge (37)
			(4) edge (38)
			(4) edge (39)
			;
			\end{tikzpicture}\]
			\caption{$G_{(J_{4,2,3},x_{4})}$}
			\label{fig:G3}
		\end{minipage}\hfill
		\begin{minipage}{0.5\textwidth}
			\centering
			\[\begin{tikzpicture}[x=1.1cm, y=1.1cm]
			
			\vertex[fill] (2) at (1,0) {};
			\vertex[fill] (3) at (0,1) {};
			\vertex[fill] (4) at (1,1) {};
			\vertex[fill] (5) at (2,0.5) {};
			\vertex[fill] (6) at (3,0.5) {};
			\vertex[fill] (7) at (-1,0.5) {};
			\vertex[fill] (8) at (-2,0.5) {};
			\vertex[fill] (9) at (0.5,2) {};
			\vertex[fill] (10) at (0.5,-1) {};
			\vertex[fill] (11) at (0.5,-2) {};
			\vertex[fill] (12) at (0.5,3) {};
			\vertex[fill] (13) at (1.3,0.3) {};
			\vertex[fill] (14) at (1.3,0.5) {};
			\vertex[fill] (15) at (1.3,0.7) {};
			\vertex[fill] (16) at (3.7,0.3) {};
			\vertex[fill] (17) at (3.7,0.5) {};
			\vertex[fill] (18) at (3.7,0.7) {};
			\vertex[fill] (19) at (0.3,1.3) {};
			\vertex[fill] (20) at (0.5,1.3) {};
			\vertex[fill] (21) at (0.7,1.3) {};
			\vertex[fill] (34) at (0.3,3.7) {};
			\vertex[fill] (35) at (0.5,3.7) {};
			\vertex[fill] (36) at (0.7,3.7) {};
			\vertex[fill] (37) at (1.65,1.65) {};
			\vertex[fill] (38) at (1.5,1.8) {};
			\vertex[fill] (39) at (1.8,1.5) {};
			\vertex[fill] (43) at (-0.65,-0.65) {};
			\vertex[fill] (44) at (-0.8,-0.5) {};
			\vertex[fill] (45) at (-0.5,-0.8) {};
			
			\path

			(4) edge (5)
			(4) edge (6)
			(4) edge (12)
			(4) edge (9)
			(9) edge (19)
			(9) edge (20)
			(9) edge (21)
			(5) edge (13)
			(5) edge (14)
			(5) edge (15)
			(12) edge (34)
			(12) edge (35)
			(12) edge (36)
			(6) edge (18)
			(6) edge (16)
			(6) edge (17)
			(4) edge (37)
			(4) edge (38)
			(4) edge (39)
			;
			\end{tikzpicture}\]
			\caption{$G_{(J_{4,2,3}:x_{4})}$}
			\label{fig:G4}
		\end{minipage}
	\end{figure}	
	\begin{Remark}\label{R1}
		{\em
			We may have the description  $S^\star_{-2,p}/\mathit{I}^\star_{-2,p}$, $S^\star_{-1,p}/\mathit{I}^\star_{-1,p}$, $S^\star_{0,p}/\mathit{I}^\star_{0,p}$, $S^\star_{-1,p,q}/\mathit{I}^\star_{-1,p,q}$ or $S^\star_{0,p,q}/\mathit{I}^\star_{0,p,q}$ while proving our results by induction on $n$. In that case we define
			\begin{itemize}
				\item $ S^\star_{-2,p}/\mathit{I}^\star_{-2,p}\cong S^\star_{-1,p}/\mathit{I}^\star_{-1,p}\cong K,$
				\item $ S^\star_{0,p}/\mathit{I}^\star_{0,p}\cong K[V(\mathcal{S}_{p})]/I(\mathcal{S}_{p})$,
             \item 
             $ S^\star_{-1,p,q}/\mathit{I}^\star_{-1,p,q}\cong K,$
		\item 
				$ S^\star_{0,p,q}/\mathit{I}^\star_{0,p,q}\cong K[V(\mathcal{S}_{p,q})]/I(\mathcal{S}_{p,q})$.
			\end{itemize} 
                 
		} 
	\end{Remark} 
	\begin{Lemma}\label{Lem1}
 {\em
		Let $n,p\geq1$. Then  $\reg(S^\star_{n,p}/\mathit{I}^\star_{n,p})=\left\lceil{\frac{n+1}{2}}\right\rceil$.
  }
	\end{Lemma}
	\begin{proof}	
We will prove this result by induction on $n$. We have the following isomorphisms:
    	\begin{equation}\label{11}
	S^\star_{n,p}/(\mathit{I}^\star_{n,p}:\mathit{x}_{n})\cong S^\star_{n-3,p}/\mathit{I}^\star_{n-3,p}\bigotimes_KK[\{\mathit{x}_{n}\}\cup\mathit{B}_{n+1}],
	\end{equation} 
  and  \begin{equation}\label{22}
	S^\star_{n,p}/(\mathit{I}^\star_{n,p},\mathit{x}_{n})\cong S^\star_{n-2,p}/\mathit{I}^\star_{n-2,p}\bigotimes_KK[V({\mathcal{S}_{2p}})]/I(\mathit{\mathcal{S}}_{2p}).
	\end{equation}
 If $n=1$, then Eq \ref{11} becomes, $S^\star_{1,p}/(\mathit{I}^\star_{1,p}:\mathit{x}_{1})\cong S^\star_{-2,p}/\mathit{I}^\star_{-2,p}\bigotimes_KK[\{x_{1}\}\cup B_2].$ By Remark \ref{R1}, $S^\star_{1,p}/(\mathit{I}^\star_{1,p}:\mathit{x}_{1})\cong K\bigotimes_KK[\{x_{1}\}\cup B_2]\cong K[\{x_{1}\}\cup B_2]$ 
 and so, $\reg(S^\star_{1,p}/(\mathit{I}^\star_{1,p}:\mathit{x}_{1}))=0.$ Now using Eq \ref{22}, we have, $S^\star_{1,p}/(\mathit{I}^\star_{1,p},\mathit{x}_{1})\cong S^\star_{-1,p}/\mathit{I}^\star_{-1,p}\bigotimes_KK[V({\mathcal{S}_{2p}})]/I(\mathit{\mathcal{S}}_{2p}),$ again by Remark \ref{R1} $$S^\star_{1,p}/(\mathit{I}^\star_{1,p},\mathit{x}_{1})\cong K\bigotimes_KK[V({\mathcal{S}_{2p}})]/I(\mathit{\mathcal{S}}_{2p})\cong K[V({\mathcal{S}_{2p}})]/I(\mathit{\mathcal{S}}_{2p}).$$ Using Theorem \ref{star}(b),   $\reg(S^\star_{1,p}/(\mathit{I}^\star_{1,p},\mathit{x}_{1}))=1.$ By Theorem \ref{reg}(c), we get the desired result, that is, $\reg(S^\star_{1,p}/\mathit{I}^\star_{1,p})=1=\left\lceil{\frac{1+1}{2}}\right\rceil.$\\ If $n=2$, then Eq \ref{11} has the form, $S^\star_{2,p}/(\mathit{I}^\star_{2,p}:\mathit{x}_{2})\cong S^\star_{-1,p}/\mathit{I}^\star_{-1,p}\bigotimes_KK[\{x_{2}\}\cup B_3].$ By Remark \ref{R1}, $S^\star_{2,p}/(\mathit{I}^\star_{2,p}:\mathit{x}_{2})\cong K\bigotimes_KK[\{x_{2}\}\cup B_3]\cong K[\{x_{2}\}\cup B_3]$ and $\reg(S^\star_{2,p}/(\mathit{I}^\star_{2,p}:\mathit{x}_{2}))=0.$ By Eq \ref{22}, we get, $S^\star_{2,p}/(\mathit{I}^\star_{2,p},\mathit{x}_{2})\cong S^\star_{0,p}/\mathit{I}^\star_{0,p}\bigotimes_KK[V({\mathcal{S}_{2p}})]/I(\mathit{\mathcal{S}}_{2p}),$ which by using Remark \ref{R1} follows $$S^\star_{2,p}/(\mathit{I}^\star_{2,p},\mathit{x}_{2})\cong K[V(\mathcal{S}_{p})]/I(\mathcal{S}_{p})\bigotimes_KK[V({\mathcal{S}_{2p}})]/I(\mathit{\mathcal{S}}_{2p}).$$ Using Lemma \ref{circulentt}, we have, $\reg(S^\star_{2,p}/(\mathit{I}^\star_{2,p},\mathit{x}_{2}))= \reg(K[V(\mathcal{S}_{p})]/I(\mathcal{S}_{p}))+\reg(K[V({\mathcal{S}_{2p}})]/I(\mathit{\mathcal{S}}_{2p})).$ By Theorem \ref{star}(b), $\reg(S^\star_{2,p}/(\mathit{I}^\star_{2,p},\mathit{x}_{2}))=1+1=2.$ The required result follows by Theorem \ref{reg}(c), that is, $\reg(S^\star_{2,p}/\mathit{I}^\star_{2,p})=2=\left\lceil{\frac{2+1}{2}}\right\rceil$. Now let $n\geq3$, we will prove the result by induction on $n$. By Eq \ref{11} and Lemma \ref{ref6}(b), $\reg(S^\star_{n,p}/(\mathit{I}^\star_{n,p}:\mathit{x}_{n}))=\reg(S^\star_{n-3,p}/\mathit{I}^\star_{n-3,p}),$ so by induction, $\reg(S^\star_{n,p}/(\mathit{I}^\star_{n,p}:\mathit{x}_{n}))=\left\lceil{\frac{n-3+1}{2}}\right\rceil=\left\lceil{\frac{n-2}{2}}\right\rceil.$ By applying Lemma \ref{circulentt} on Eq \ref{22}  	$$\reg\big(S^\star_{n,p}/(\mathit{I}^\star_{n,p},\mathit{x}_{n})\big)=\reg(S^\star_{n-2,p}/\mathit{I}^\star_{n-2,p})+\reg(K[V({\mathcal{S}}_{2p})]/\mathit{I(\mathcal{S}}_{2p})).$$ 
 By induction and Theorem \ref{star}(b),
$\reg(S^\star_{n,p}/(\mathit{I}^\star_{n,p},\mathit{x}_{n}))=\left\lceil{\frac{n-2+1}{2}}\right\rceil+1=\left\lceil{\frac{n+1}{2}}\right\rceil.$ Hence by Theorem \ref{reg}(c), $\reg(S^\star_{n,p}/\mathit{I}^\star_{n,p})=\left\lceil{\frac{n+1}{2}}\right\rceil.$	
		
	\end{proof}
	\begin{Theorem}\label{Th1}
 {\em
		Let $n,p\geq 1$. Then $\reg(\mathit{S}_{n,p}/\mathit{I}_{n,p})=\left\lceil{\frac{n+1}{2}}\right\rceil$.
  }
	\end{Theorem}
	\begin{proof}
The result will be proved by induction on $n$. Looking at the structure of the graph it is easy to see that we have the following isomorphisms:
  \begin{equation}\label{33}
\mathit{S}_{n,p}/(\mathit{I}_{n,p}:\mathit{x}_{n})\cong \mathit{S}^\star_{n-3,p}/\mathit{I}^\star_{n-3,p}\bigotimes_KK[x_{n}],
  \end{equation} 
  and
  \begin{equation}\label{44}
  \mathit{S}_{n,p}/(\mathit{I}_{n,p},\mathit{x}_{n})\cong \mathit{S}^\star_{n-2,p}/\mathit{I}^\star_{n-2,p}\bigotimes_{K}K[V({\mathcal{S}_{p}})]/I(\mathit{\mathcal{S}}_{p}).
  \end{equation}
		If $n=1$, then by Eq \ref{33}, we have $\mathit{S}_{1,p}/(\mathit{I}_{1,p}:\mathit{x}_{1})\cong (\mathit{S}^\star_{-2,p}/\mathit{I}^\star_{-2,p})\bigotimes_KK[x_{1}],$ which by using Remark \ref{R1} implies that $\mathit{S}_{1,p}/(\mathit{I}_{1,p}:\mathit{x}_{1})\cong K\bigotimes_KK[x_{1}]\cong K[x_{1}],$ so we get, $\reg(S_{1,p}/(\mathit{I}_{1,p}:\mathit{x}_{1}))=0.$ Now by Eq \ref{44},  $\mathit{S}_{1,p}/(\mathit{I}_{1,p},\mathit{x}_{1})\cong \mathit{S}^\star_{-1,p}/\mathit{I}^\star_{-1,p}\bigotimes_KK[V({\mathcal{S}_{p}})]/I(\mathit{\mathcal{S}}_{p})$ and by Remark \ref{R1} $$\mathit{S}_{1,p}/(\mathit{I}_{1,p},\mathit{x}_{1})\cong K\bigotimes_KK[V({\mathcal{S}_{p}})]/I(\mathit{\mathcal{S}}_{p})\cong K[V({\mathcal{S}_{p}})]/I(\mathit{\mathcal{S}}_{p}).$$ By Theorem \ref{star}(b), $\reg(S_{1,p}/(\mathit{I}_{1,p},\mathit{x}_{1}))=1.$ The required result follows by Theorem \ref{reg}(c), that is, $\reg(S_{1,p}/\mathit{I}_{1,p})=1=\left\lceil{\frac{1+1}{2}}\right\rceil.$ Similarly, if $n=2,$ then the desired result can easily be verified, that is, $\reg(S_{2,p}/\mathit{I}_{2,p})=2=\left\lceil{\frac{2+1}{2}}\right\rceil$. Let $n\geq3$. Applying Lemma \ref{ref6}(b) on  Eq \ref{33}, we have $\reg(\mathit{S}_{n,p}/(\mathit{I}_{n,p}:\mathit{x}_{n}))= \reg(\mathit{S}^\star_{n-3,p}/\mathit{I}^\star_{n-3,p}).$ By using Lemma \ref{Lem1} we have $$\reg(\mathit{S}_{n,p}/(\mathit{I}_{n,p}:\mathit{x}_{n}))=\left\lceil{\frac{n-3+1}{2}}\right\rceil=\left\lceil{\frac{n-2}{2}}\right\rceil.$$ Applying Lemma \ref{circulentt} on Eq \ref{44} we get $$\reg\big(S_{n,p}/(\mathit{I}_{n,p},\mathit{x}_{n})\big)=\reg(S^\star_{n-2,p}/\mathit{I}^\star_{n-2,p})+\reg(K[V({\mathcal{S}}_{p})]/\mathit{I(\mathcal{S}}_{p})).$$ By Lemma \ref{Lem1} and Theorem \ref{star}(b), we have
$\reg(\mathit{S}_{n,p}/(\mathit{I}_{n,p},\mathit{x}_{n}))=\left\lceil{\frac{n-2+1}{2}}\right\rceil+1=\left\lceil{\frac{n+1}{2}}\right\rceil.$
	    Hence by Theorem \ref{reg}(c),
		$\reg(\mathit{S}_{n,p}/(\mathit{I}_{n,p}))=\left\lceil{\frac{n+1}{2}}\right\rceil.$
	\end{proof}
	\begin{Lemma}\label{Lem2}
		{\em
			Let $n,p,q\geq1$. Then 
			\begin{itemize}
				\item[(a)] $\depth(S^\star_{n,p,q}/\mathit{I}^\star_{n,p,q})=\sdepth(S^\star_{n,p,q}/\mathit{I}^\star_{n,p,q})=(p+q)(n+1),$
				\item[(b)]$\reg(S^\star_{n,p,q}/\mathit{I}^\star_{n,p,q})=(n+1)p,$
				\item[(c)]$\pdim(S^\star_{n,p,q}/\mathit{I}^\star_{n,p,q})=(1+pq)(n+1)$.
				
			\end{itemize} 
		} 
	\end{Lemma}
	\begin{proof}
		First we prove the result for depth and Stanley depth by induction on $n$. Consider the short exact sequence
		\begin{equation}\label{es1}
		0\longrightarrow S^\star_{n,p,q}/(\mathit{I}^\star_{n,p,q}:\mathit{x}_{n+1})\xrightarrow{\cdot \mathit{x}_{n+1}} S^\star_{n,p,q}/\mathit{I}^\star_{n,p,q}\longrightarrow S^\star_{n,p,q}/(\mathit{I}^\star_{n,p,q},\mathit{x}_{n+1})\longrightarrow 0,
		\end{equation}
		by applying Depth Lemma and Lemma \ref{sdepth} on Eq \ref{es1}, we get
		\begin{equation}\label{d}
\depth(S^\star_{n,p,q}/\mathit{I}^\star_{n,p,q})\geq \min\{\depth(S^\star_{n,p,q}/(\mathit{I}^\star_{n,p,q}:\mathit{x}_{n+1})), \depth(S^\star_{n,p,q}/(\mathit{I}^\star_{n,p,q},\mathit{x}_{n+1}))\},
		\end{equation}
        and
        \begin{equation}\label{ssd}
\sdepth(S^\star_{n,p,q}/\mathit{I}^\star_{n,p,q})\geq \min\{\sdepth(S^\star_{n,p,q}/(\mathit{I}^\star_{n,p,q}:\mathit{x}_{n+1})), \sdepth(S^\star_{n,p,q}/(\mathit{I}^\star_{n,p,q},\mathit{x}_{n+1}))\}.
		\end{equation}
		We have the following isomorphisms:
		\begin{equation}\label{55}
		S^\star_{n,p,q}/(\mathit{I}^\star_{n,p,q}:\mathit{x}_{n+1})\cong S^\star_{n-2,p,q}/\mathit{I}^\star_{n-2,p,q}\bigotimes_KK[\{\mathit{x}_{n+1}\}\cup\mathit{A}_{n}\cup\mathit{C}_{n}\cup\mathit{C}_{n+1}],
		\end{equation} 
\begin{equation}\label{66}
S^{\star}_{n,p,q}/(\mathit{I}^{\star}_{n,p,q},\mathit{x}_{n+1})\cong S^{\star}_{n-1,p,q}/\mathit{I}^{\star}_{n-1,p,q}\bigotimes_K K[\mathit{A}_{n+1}]\bigotimes_K\bigotimes_{k=1}^{p} K[V({\mathcal{S}_q})]/I\mathit{(\mathcal{S}_q)},
\end{equation} 
  and
\begin{equation}\label{low}
S^\star_{n,p,q}/(\mathit{I}^\star_{n,p,q}:\mathit{y}_{(n+1)1}\mathit{y}_{(n+1)2}\dots\mathit{y}_{(n+1)p})\cong S^\star_{n-1,p,q}/\mathit{I}^\star_{n-1,p,q}\bigotimes_KK[\mathit{A}_{n+1}\cup\mathit{B}_{n+1}].
\end{equation}
		If $n=1$, then by Eq \ref{55}, we have $S^\star_{1,p,q}/(\mathit{I}^\star_{1,p,q}:\mathit{x}_{2})\cong S^\star_{-1,p,q}/\mathit{I}^\star_{-1,p,q}\bigotimes_KK[\{x_2\}\cup A_1\cup C_1\cup C_2]$. By Remark \ref{R1}, $S^\star_{1,p,q}/(\mathit{I}^\star_{1,p,q}:\mathit{x}_{2})\cong K\bigotimes_KK[\{x_2\}\cup A_1\cup C_1\cup C_2],$ which implies that 
  \begin{equation}\label{77}
S^\star_{1,p,q}/(\mathit{I}^\star_{1,p,q}:\mathit{x}_{2})\cong K[\{x_2\}\cup A_1\cup C_1\cup C_2].
  \end{equation}
  Thus $\depth(S^\star_{1,p,q}/(\mathit{I}^\star_{1,p,q}:\mathit{x}_{2}))=1+q+pq+pq=2pq+q+1,$ and  $\sdepth(S^\star_{1,p,q}/(\mathit{I}^\star_{1,p,q}:\mathit{x}_{2}))= 1+q+pq+pq=2pq+q+1,$ Similarly, Eq \ref{66} has the form $S^\star_{1,p,q}/(\mathit{I}^\star_{1,p,q},\mathit{x}_{2})\cong S^\star_{0,p,q}/\mathit{I}^\star_{0,p,q}\bigotimes_{K}K[\mathit{A}_{2}]\bigotimes_{K}\bigotimes_{k=1}^{p}{}_{K}K[V({\mathcal{S}_q})]/I(\mathit{\mathcal{S}_q})$ and by Remark \ref{R1}
  \begin{equation}\label{88}
S^\star_{1,p,q}/(\mathit{I}^\star_{1,p,q},\mathit{x}_{2})\cong K[V(\mathcal{S}_{p,q})]/I(\mathcal{S}_{p,q})\bigotimes_KK[\mathit{A}_{2}]\bigotimes_K\bigotimes_{k=1}^{p}{}_KK[V({\mathcal{S}_q})]/I\mathit{(\mathcal{S}_q)}.
\end{equation}

Using Lemma \ref{ref3}, Lemma \ref{ref5} and Lemma \ref{ref6} (a)$$\depth(S^\star_{1,p,q}/(\mathit{I}^\star_{1,p,q},x_{2}))=\depth(K[V(\mathcal{S}_{p,q})]/I(\mathcal{S}_{p,q}))+\depth(K[\mathit{A}_{2}])+\\\sum_{k=1}^{p}\depth(K[V({\mathcal{S}_q})]/I\mathit{(\mathcal{S}_q)}),$$
and 
\begin{align*}
\sdepth(S^\star_{1,p,q}/(\mathit{I}^\star_{1,p,q},x_{2}))\geq\sdepth(K[V(\mathcal{S}_{p,q})]/I(\mathcal{S}_{p,q}))+\sdepth(K[\mathit{A}_{2}])+\\\sum_{k=1}^{p}\sdepth(K[V({\mathcal{S}_q})]/I\mathit{(\mathcal{S}_q)}).    
\end{align*}
By Corollary \ref{bstar}(a) and Theorem \ref{star}(a) $$\depth(S^\star_{1,p,q}/(\mathit{I}^\star_{1,p,q},\mathit{x}_{2}))=p+q+q+\sum_{k=1}^{p}1=2q+p+p=2(p+q),$$
and
$$\sdepth(S^\star_{1,p,q}/(\mathit{I}^\star_{1,p,q},\mathit{x}_{2}))\geq p+q+q+\sum_{k=1}^{p}1=2q+p+p=2(p+q).$$
It follows by Eq \ref{d} and Eq \ref{ssd} that, $\depth(S^\star_{1,p,q}/\mathit{I}^\star_{1,p,q}),\sdepth(S^\star_{1,p,q}/\mathit{I}^\star_{1,p,q})\geq2(p+q).$ Now since $\mathit{y}_{21}\mathit{y}_{22}\dots\mathit{y}_{2p}\notin\mathit{I}^\star_{1,p,q},$ so by Eq \ref{low}, we have $S^\star_{1,p,q}/(\mathit{I}^\star_{1,p,q}:\mathit{y}_{21}\mathit{y}_{22}\dots\mathit{y}_{2p})\cong S^\star_{0,p,q}/\mathit{I}^\star_{0,p,q}\bigotimes_KK[\mathit{A}_{2}\cup\mathit{B}_{2}]$ 
  and using Remark \ref{R1}, $S^\star_{1,p,q}/(\mathit{I}^\star_{1,p,q}:\mathit{y}_{21}\mathit{y}_{22}\dots\mathit{y}_{2p})\cong K[V(\mathcal{S}_{p,q})]/I(\mathcal{S}_{p,q})\bigotimes_KK[\mathit{A}_{2}\cup\mathit{B}_{2}].$ By Lemma \ref{ref6}(a) $$\depth(S^\star_{1,p,q}/(\mathit{I}^\star_{1,p,q}:\mathit{y}_{21}\mathit{y}_{22}\dots\mathit{y}_{2p}))=\depth(K[V(\mathcal{S}_{p,q})]/I(\mathcal{S}_{p,q}))+\depth(K[\mathit{A}_{2}\cup\mathit{B}_{2}]),$$
  and 
$$\sdepth(S^\star_{1,p,q}/(\mathit{I}^\star_{1,p,q}:\mathit{y}_{21}\mathit{y}_{22}\dots\mathit{y}_{2p}))=\sdepth(K[V(\mathcal{S}_{p,q})]/I(\mathcal{S}_{p,q}))+\sdepth(K[\mathit{A}_{2}\cup\mathit{B}_{2}]).$$
  Using Corollary \ref{bstar}(a), we get $\depth(S^\star_{1,p,q}/(\mathit{I}^\star_{1,p,q}:y_{21},y_{21}\mathit{y}_{22}\dots\mathit{y}_{2p}))=p+q+p+q=2(p+q),$ and $\sdepth(S^\star_{1,p,q}/(\mathit{I}^\star_{1,p,q}:y_{21},y_{21}\mathit{y}_{22}\dots\mathit{y}_{2p}))=p+q+p+q=2(p+q).$ By using Corollary \ref{c1} and Proposition \ref{c2}, we get $\depth(S^\star_{1,p,q}/\mathit{I}^\star_{1,p,q})\leq2(p+q),$ and $\sdepth(S^\star_{1,p,q}/\mathit{I}^\star_{1,p,q})\leq2(p+q).$ Therefore, $\depth(S^\star_{1,p,q}/\mathit{I}^\star_{1,p,q})=\sdepth(S^\star_{1,p,q}/\mathit{I}^\star_{1,p,q})=2(p+q).$ If $n=2$, then using the similar arguments and case $n=1$, one can easily prove that $\depth(S^\star_{2,p,q}/\mathit{I}^\star_{2,p,q})=\sdepth(S^\star_{2,p,q}/\mathit{I}^\star_{2,p,q})=3(p+q).$ 
Now let $n\geq3$. By Eq \ref{55} and Lemma \ref{ref6}(a) , it follows that $$\depth(S^\star_{n,p,q}/(\mathit{I}^\star_{n,p,q}:\mathit{x}_{n+1}))=\depth(S^\star_{n-2,p,q}/\mathit{I}^\star_{n-2,p,q})+\depth(K[\{\mathit{x}_{n+1}\}\cup\mathit{A}_{n}\cup\mathit{C}_{n}\cup\mathit{C}_{n+1}]),$$
and 
$$\sdepth(S^\star_{n,p,q}/(\mathit{I}^\star_{n,p,q}:\mathit{x}_{n+1}))=\sdepth(S^\star_{n-2,p,q}/\mathit{I}^\star_{n-2,p,q})+\sdepth(K[\{\mathit{x}_{n+1}\}\cup\mathit{A}_{n}\cup\mathit{C}_{n}\cup\mathit{C}_{n+1}]).$$
By induction $$\depth(S^\star_{n,p,q}/(\mathit{I}^\star_{n,p,q}:\mathit{x}_{n+1}))=(p+q)(n-2+1)+2pq+q+1=(p+q)n+2pq-p+1,$$
and 
$$\sdepth(S^\star_{n,p,q}/(\mathit{I}^\star_{n,p,q}:\mathit{x}_{n+1}))=(p+q)(n-2+1)+2pq+q+1=(p+q)n+2pq-p+1.$$
		Using Lemma \ref{ref3}, Lemma \ref{ref5} and Lemma \ref{ref6} (a) on Eq \ref{66}
\begin{align*}\depth\big(S^\star_{n,p,q}/(\mathit{I}^\star_{n,p,q},\mathit{x}_{n+1})\big
		)=& \depth(S^\star_{n-1,p,q}/\mathit{I}^\star_{n-1,p,q})+\depth(K[A_{n+1}])\\
  &+\sum_{k=1}^{p}\depth(K[V({\mathcal{S}_q})]/\mathit{I(\mathcal{S}_q)}),
  \end{align*}
  and
  \begin{align*}\sdepth\big(S^\star_{n,p,q}/(\mathit{I}^\star_{n,p,q},\mathit{x}_{n+1})\big
		)\geq& \sdepth(S^\star_{n-1,p,q}/\mathit{I}^\star_{n-1,p,q})+\sdepth(K[A_{n+1}])\\
  &+\sum_{k=1}^{p}\sdepth(K[V({\mathcal{S}_q})]/\mathit{I(\mathcal{S}_q)}).
  \end{align*}
		By induction and Theorem \ref{star}(a)
$$\depth(S^\star_{n,p,q}/(\mathit{I}^\star_{n,p,q},\mathit{x}_{n+1}))=(p+q)(n-1+1)+q+\sum_{k=1}^{p}1=(p+q)n+q+p=(p+q)(n+1),$$
and
$$\sdepth(S^\star_{n,p,q}/(\mathit{I}^\star_{n,p,q},\mathit{x}_{n+1}))\geq(p+q)(n-1+1)+q+\sum_{k=1}^{p}1=(p+q)n+q+p=(p+q)(n+1).$$
	 Again by using Eq \ref{d} and Eq \ref{ssd}, $\depth(S^\star_{n,p,q}/\mathit{I}^\star_{n,p,q}),\sdepth(S^\star_{n,p,q}/\mathit{I}^\star_{n,p,q})\geq (p+q)(n+1).$
		Since $\mathit{y}_{(n+1)1}\mathit{y}_{(n+1)2}\dots\mathit{y}_{(n+1)p}\notin\mathit{I}^\star_{n,p,q},$ so by Eq \ref{low} and Lemma \ref{ref6}(a)
\begin{align*}
\depth(S^\star_{n,p,q}/(\mathit{I}^\star_{n,p,q}:\mathit{y}_{(n+1)1}\mathit{y}_{(n+1)2}\dots\mathit{y}_{(n+1)p}))=&\depth(S^\star_{n-1,p,q}/\mathit{I}^\star_{n-1,p,q})\\
&+\depth(K[\mathit{A}_{n+1}\cup\mathit{B}_{n+1}]),   \end{align*}
and
\begin{align*}\sdepth(S^\star_{n,p,q}/(\mathit{I}^\star_{n,p,q}:\mathit{y}_{(n+1)1}\mathit{y}_{(n+1)2}\dots\mathit{y}_{(n+1)p}))&=\sdepth(S^\star_{n-1,p,q}/\mathit{I}^\star_{n-1,p,q})\\
&+\sdepth(K[\mathit{A}_{n+1}\cup\mathit{B}_{n+1}]).\end{align*}
By induction	$$\depth(S^\star_{n,p,q}/(\mathit{I}^\star_{n,p,q}:\mathit{y}_{(n+1)1}\mathit{y}_{(n+1)2}\dots\mathit{y}_{(n+1)p}))=(p+q)(n-1+1)+p+q=(p+q)(n+1),$$
and $$\sdepth(S^\star_{n,p,q}/(\mathit{I}^\star_{n,p,q}:\mathit{y}_{(n+1)1}\mathit{y}_{(n+1)2}\dots\mathit{y}_{(n+1)p}))=(p+q)(n-1+1)+p+q=(p+q)(n+1).$$
		Again by Corollary \ref{c1} and Proposition \ref{c2},  $\depth(S^\star_{n,p,q}/\mathit{I}^\star_{n,p,q}),\sdepth(S^\star_{n,p,q}/\mathit{I}^\star_{n,p,q})\leq (p+q)(n+1).$ Hence $\depth(S^\star_{n,p,q}/\mathit{I}^\star_{n,p,q})=\sdepth(S^\star_{n,p,q}/\mathit{I}^\star_{n,p,q})=(p+q)(n+1)$, as desired.\\
		\indent Now we prove the result for regularity by induction on $n$. If $n=1$, then by Eq \ref{77}, we have		$\reg(S^\star_{1,p,q}/(\mathit{I}^\star_{1,p,q}:\mathit{x}_{2}))=\reg(K[\{x_2\}\cup A_1\cup C_1\cup C_2])=0.$ Applying Lemma \ref{circulentt} and Lemma \ref{ref6}(b) on  Eq \ref{88}  $$\reg(S^\star_{1,p,q}/(\mathit{I}^\star_{1,p,q},\mathit{x}_{2}))= \reg(K[V(\mathcal{S}_{p,q})]/I(\mathcal{S}_{p,q}))+\sum^{p}_{k=1}\reg(K[V({\mathcal{S}_q})]/I\mathit{(\mathcal{S}_q)}).$$ 
   Using Corollary \ref{bstar}(b) and Theorem \ref{star}(b), $\reg\big(S^\star_{1,p,q}/(\mathit{I}^\star_{1,p,q},x_{2})\big)=p+\sum_{k=1}^{p}1=p+p=2p.$ The required result follows by Theorem \ref{reg}(c), that is $\reg\big(S^\star_{1,p,q}/\mathit{I}^\star_{1,p,q}\big)=2p.$ If $n=2$, then by using the similar arguments and case $n=1$, we get the desired result, that is $\reg\big(S^\star_{2,p,q}/\mathit{I}^\star_{2,p,q}\big)=3p.$ Now let $n\geq3$. By Eq \ref{55} and Lemma \ref{ref6}(b), $\reg(S^\star_{n,p,q}/(\mathit{I}^\star_{n,p,q}:\mathit{x}_{n+1}))=\reg(S^\star_{n-2,p,q}/\mathit{I}^\star_{n-2,p,q}).$ So by induction, $\reg(S^\star_{n,p,q}/(\mathit{I}^\star_{n,p,q}:\mathit{x}_{n+1}))=(n-2+1)p=(n-1)p.$
		Applying Lemma \ref{circulentt} and Lemma \ref{ref6}(b) on  Eq \ref{66}, we get	$$\reg\big(S^\star_{n,p,q}/(\mathit{I}^\star_{n,p,q},\mathit{x}_{n+1})\big
		)= \reg(S^\star_{n-1,p,q}/\mathit{I}^\star_{n-1,p,q})+\sum_{k=1}^{p}\reg(K[V({\mathcal{S}_q})]/\mathit{I(\mathcal{S}_q)}).$$  Again by induction and Theorem \ref{star}(b)
$$\reg(S^\star_{n,p,q}/(\mathit{I}^\star_{n,p,q},\mathit{x}_{n+1}))=(n-1+1)p+\sum_{k=1}^{p}1=np+p=(n+1)p.$$	Hence by Theorem \ref{reg}(c), 
		$\reg(S^\star_{n,p,q}/\mathit{I}^\star_{n,p,q})=(n+1)p.$ 

  The result for projective dimension follows by using Auslander–Buchsbaum formula, that is ${\pdim}(S^\star_{n,p,q}/\mathit{I}^\star_{n,p,q})+{\depth}(S^\star_{n,p,q}/\mathit{I}^\star_{n,p,q})={\depth}(S^\star_{n,p,q}),$ which implies that ${\pdim}(S^\star_{n,p,q}/\mathit{I}^\star_{n,p,q})={\depth}(S^\star_{n,p,q})-{\depth}(S^\star_{n,p,q}/\mathit{I}^\star_{n,p,q}).$ Hence
  $${\pdim}(S^\star_{n,p,q}/\mathit{I}^\star_{n,p,q})=(1+p+q+pq)(n+1)-(p+q)(n+1)=(1+pq)(n+1).$$

		
		
	\end{proof}

	\begin{Theorem}\label{Th2}
		{\em
			Let $n,p,q\geq 1$. Then
			\begin{itemize}
				\item[(a)] $\depth(\mathit{S}_{n,p,q}/\mathit{I}_{n,p,q})=\sdepth(\mathit{S}_{n,p,q}/\mathit{I}_{n,p,q})=(p+q)n+q.$
				\item[(b)]$\reg(\mathit{S}_{n,p,q}/\mathit{I}_{n,p,q})=np.$
				\item[(c)]$\pdim(\mathit{S}_{n,p,q}/\mathit{I}_{n,p,q})=(1+pq)n+1$.
				
			\end{itemize}
		}  
	\end{Theorem}
	\begin{proof}
		First we prove the result for depth. Consider the short exact sequence
		\begin{equation}\label{es3}
		0\longrightarrow \mathit{S}_{n,p,q}/(\mathit{I}_{n,p,q}:\mathit{x}_{n+1})\xrightarrow{\cdot \mathit{x}_{n+1}} \mathit{S}_{n,p,q}/\mathit{I}_{n,p,q}\longrightarrow \mathit{S}_{n,p,q}/(\mathit{I}_{n,p,q},\mathit{x}_{n+1})\longrightarrow 0,
		\end{equation}
		by applying Depth Lemma and Lemma \ref{sdepth} on Eq \ref{es3}, we have
		\begin{equation}\label{e}
\depth(\mathit{S}_{n,p,q}/\mathit{I}_{n,p,q})\geq \min\{\depth(\mathit{S}_{n,p,q}/(\mathit{I}_{n,p,q}:\mathit{x}_{n+1})), \depth(\mathit{S}_{n,p,q}/(\mathit{I}_{n,p,q},\mathit{x}_{n+1}))\}.
		\end{equation}
        and
        \begin{equation}\label{sse}
\sdepth(\mathit{S}_{n,p,q}/\mathit{I}_{n,p,q})\geq \min\{\sdepth(\mathit{S}_{n,p,q}/(\mathit{I}_{n,p,q}:\mathit{x}_{n+1})), \sdepth(\mathit{S}_{n,p,q}/(\mathit{I}_{n,p,q},\mathit{x}_{n+1}))\}.
		\end{equation}
		We have the following isomorphisms:
	\begin{equation}\label{99}
	\mathit{S}_{n,p,q}/(\mathit{I}_{n,p,q}:\mathit{x}_{n+1})\cong \mathit{S}^\star_{n-2,p,q}/\mathit{I}^\star_{n-2,p,q}\bigotimes_KK[\{\mathit{x}_{n+1}\}
\cup\mathit{A}_{n}\cup\mathit{C}_{n}],
\end{equation}
    \begin{equation}\label{1010}
    \mathit{S}_{n,p,q}/(\mathit{I}_{n,p,q},\mathit{x}_{n+1})\cong \mathit{S}^\star_{n-1,p,q}/\mathit{I}^\star_{n-1,p,q}\bigotimes_KK[\mathit{A}_{n+1}],
    \end{equation}
and
\begin{equation}\label{333}
\mathit{S}_{n,p,q}/(\mathit{I}_{n,p,q}:\mathit{x}_{(n+1)1}\mathit{x}_{(n+1)2}\dots\mathit{x}_{(n+1)q})\cong \mathit{S}^\star_{n-1,p,q}/\mathit{I}^\star_{n-1,p,q}\bigotimes_KK[\mathit{A}_{n+1}]. 
\end{equation}
    If $n=1$, then Eq \ref{99} becomes $S_{1,p,q}/(\mathit{I}_{1,p,q}:\mathit{x}_{2})\cong S^\star_{-1,p,q}/\mathit{I}^\star_{-1,p,q}\bigotimes_KK[\{x_2\}\cup A_1\cup C_1]$ and by Remark \ref{R1}, $S_{1,p,q}/(\mathit{I}_{1,p,q}:\mathit{x}_{2})\cong K\bigotimes_KK[\{x_2\}\cup A_1\cup C_1],$ which implies that 
  \begin{equation}\label{111}
S_{1,p,q}/(\mathit{I}_{1,p,q}:\mathit{x}_{2})\cong K[\{x_2\}\cup A_1\cup C_1].
  \end{equation}
  Thus $\depth(S_{1,p,q}/(\mathit{I}_{1,p,q}:\mathit{x}_{2}))=\sdepth(S_{1,p,q}/(\mathit{I}_{1,p,q}:\mathit{x}_{2}))=1+q+pq$ and by Eq \ref{1010}, $   \mathit{S}_{1,p,q}/(\mathit{I}_{1,p,q},\mathit{x}_{2})\cong \mathit{S}^\star_{0,p,q}/\mathit{I}^\star_{0,p,q}\bigotimes_KK[\mathit{A}_{2}].$ By Remark \ref{R1}  
  \begin{equation}\label{222}
  \mathit{S}_{1,p,q}/(\mathit{I}_{1,p,q},\mathit{x}_{2})\cong K[V(\mathcal{S}_{p,q})]/I(\mathcal{S}_{p,q})\bigotimes_KK[\mathit{A}_{2}].
  \end{equation}
  By Corollary \ref{bstar}(a) and Lemma \ref{ref6}(a), $\depth(\mathit{S}_{1,p,q}/(\mathit{I}_{1,p,q},\mathit{x}_{2}))=\sdepth(\mathit{S}_{1,p,q}/(\mathit{I}_{1,p,q},\mathit{x}_{2}))=p+q+q=2q+p.$ By Eq \ref{e} and Eq \ref{sse}  $\depth(\mathit{S}_{1,p,q}/\mathit{I}_{1,p,q})\geq 2q+p,$ and $\sdepth(\mathit{S}_{1,p,q}/\mathit{I}_{1,p,q})\geq 2q+p.$ Now since $\mathit{x}_{21}\mathit{x}_{22}\dots\mathit{x}_{2q}\notin\mathit{I}_{1,p,q},$ so by Eq \ref{333} and Remark \ref{R1}, $\mathit{S}_{1,p,q}/(\mathit{I}_{1,p,q}:\mathit{x}_{21}\mathit{x}_{22}\dots\mathit{x}_{2q})\cong K[V(\mathcal{S}_{p,q})]/I(\mathcal{S}_{p,q})\bigotimes_KK[\mathit{A}_{2}],$ thus by Corollary \ref{bstar}(a) and Lemma \ref{ref6}(a), $\depth(\mathit{S}_{1,p,q}/(\mathit{I}_{1,p,q}:\mathit{x}_{21}\mathit{x}_{22}\dots\mathit{x}_{2q}))=\sdepth(\mathit{S}_{1,p,q}/(\mathit{I}_{1,p,q}:\mathit{x}_{21}\mathit{x}_{22}\dots\mathit{x}_{2q}))=2q+p.$
		Since by Corollary \ref{c1} and Proposition \ref{c2}, $\depth(\mathit{S}_{1,p,q}/\mathit{I}_{1,p,q})\leq 2q+p$ and $\sdepth(\mathit{S}_{1,p,q}/\mathit{I}_{1,p,q})\leq 2q+p.$ Hence $\depth(\mathit{S}_{1,p,q}/\mathit{I}_{1,p,q})=\sdepth(\mathit{S}_{1,p,q}/\mathit{I}_{1,p,q})=2q+p.$ If $n=2$, then using the similar arguments and case $n=1$, we find the desired value, that is, $\depth(S_{2,p,q}/\mathit{I}_{2,p,q})=\sdepth(S_{2,p,q}/\mathit{I}_{2,p,q})=2p+3q.$
   Now let $n\geq3$. Using Eq \ref{99} and Lemma \ref{ref6}(a), we have
$$\depth(\mathit{S}_{n,p,q}/(\mathit{I}_{n,p,q}:\mathit{x}_{n+1}))=\depth(\mathit{S}^\star_{n-2,p,q}/\mathit{I}^\star_{n-2,p,q})+\depth(K[\{\mathit{x}_{n+1}\}
\cup\mathit{A}_{n}\cup\mathit{C}_{n}]),$$
and
$$\sdepth(\mathit{S}_{n,p,q}/(\mathit{I}_{n,p,q}:\mathit{x}_{n+1}))=\sdepth(\mathit{S}^\star_{n-2,p,q}/\mathit{I}^\star_{n-2,p,q})+\sdepth(K[\{\mathit{x}_{n+1}\}
\cup\mathit{A}_{n}\cup\mathit{C}_{n}]).$$ By Lemma \ref{Lem2}(a), $\depth(\mathit{S}_{n,p,q}/(\mathit{I}_{n,p,q}:\mathit{x}_{n+1}))=\sdepth(\mathit{S}_{n,p,q}/(\mathit{I}_{n,p,q}:\mathit{x}_{n+1}))=(p+q)(n-2+1)+pq+q+1=(p+q)n+pq-p+1,$ and by Eq \ref{1010} and Lemma \ref{ref6}(a) $$\depth(\mathit{S}_{n,p,q}/(\mathit{I}_{n,p,q},\mathit{x}_{n+1}))=\depth(\mathit{S}^\star_{n-1,p,q}/\mathit{I}^\star_{n-1,p,q})+\depth(K[\mathit{A}_{n+1}]),$$
and
$$\sdepth(\mathit{S}_{n,p,q}/(\mathit{I}_{n,p,q},\mathit{x}_{n+1}))=\sdepth(\mathit{S}^\star_{n-1,p,q}/\mathit{I}^\star_{n-1,p,q})+\sdepth(K[\mathit{A}_{n+1}]).$$
Again by Lemma \ref{Lem2}(a), we get,
$\depth(\mathit{S}_{n,p,q}/(\mathit{I}_{n,p,q},\mathit{x}_{n+1}))=\sdepth(\mathit{S}_{n,p,q}/(\mathit{I}_{n,p,q},\mathit{x}_{n+1}))=(p+q)(n-1+1)+q=(p+q)n+q.$
    	Thus by Eq \ref{e} and Eq \ref{sse}, $\depth(\mathit{S}_{n,p,q}/\mathit{I}_{n,p,q})\geq (p+q)n+q,$ and $\sdepth(\mathit{S}_{n,p,q}/\mathit{I}_{n,p,q})\geq (p+q)n+q.$ Now since $\mathit{x}_{(n+1)1}\mathit{x}_{(n+1)2}\dots\mathit{x}_{(n+1)q}\notin\mathit{I}_{n,p,q},$ so by Eq \ref{333} and Lemma \ref{ref6}(a) $$\depth(\mathit{S}_{n,p,q}/(\mathit{I}_{n,p,q}:\mathit{x}_{(n+1)1}\mathit{x}_{(n+1)2}\dots\mathit{x}_{(n+1)q}))=\depth(\mathit{S}^\star_{n-1,p,q}/\mathit{I}^\star_{n-1,p,q})+\depth(K[\mathit{A}_{n+1}]),$$
        and       $$\sdepth(\mathit{S}_{n,p,q}/(\mathit{I}_{n,p,q}:\mathit{x}_{(n+1)1}\mathit{x}_{(n+1)2}\dots\mathit{x}_{(n+1)q}))=\sdepth(\mathit{S}^\star_{n-1,p,q}/\mathit{I}^\star_{n-1,p,q})+\sdepth(K[\mathit{A}_{n+1}]).$$
Using Lemma \ref{Lem2}(a) $$\depth(\mathit{S}_{n,p,q}/(\mathit{I}_{n,p,q}:\mathit{x}_{(n+1)1}\mathit{x}_{(n+1)2}\dots\mathit{x}_{(n+1)q}))= (p+q)(n-1+1)+q=(p+q)n+q,$$
and
$$\sdepth(\mathit{S}_{n,p,q}/(\mathit{I}_{n,p,q}:\mathit{x}_{(n+1)1}\mathit{x}_{(n+1)2}\dots\mathit{x}_{(n+1)q}))= (p+q)(n-1+1)+q=(p+q)n+q.$$
By Corollary \ref{c1} and Proposition \ref{c2}, $\depth(\mathit{S}_{n,p,q}/\mathit{I}_{n,p,q}),\sdepth(\mathit{S}_{n,p,q}/\mathit{I}_{n,p,q})\leq (p+q)n+q.$ Hence $\depth(\mathit{S}_{n,p,q}/\mathit{I}_{n,p,q})=\sdepth(\mathit{S}_{n,p,q}/\mathit{I}_{n,p,q})=(p+q)n+q.$

 If $n=1$, then by Eq \ref{111}, we have $\reg(S_{1,p,q}/(\mathit{I}_{1,p,q}:\mathit{x}_{2}))=\reg(K[\{x_2\}\cup A_1\cup C_1])=0.$ By Eq \ref{222} and Lemma \ref{ref6}(b) $$ \reg(\mathit{S}_{1,p,q}/(\mathit{I}_{1,p,q},\mathit{x}_{2}))=\reg(K[V(\mathcal{S}_{p,q})]/I(\mathcal{S}_{p,q})).$$ Using Corollary \ref{bstar}(b),  $\reg(\mathit{S}_{1,p,q}/(\mathit{I}_{1,p,q},\mathit{x}_{2}))=p.$ So by Theorem \ref{reg}(c), $\reg(\mathit{S}_{1,p,q}/\mathit{I}_{1,p,q})=p.$ If $n=2$, then by using the similar arguments and case $n=1$, the result can be easily verified that is $\reg(\mathit{S}_{2,p,q}/\mathit{I}_{2,p,q})=2p.$ Now let $n\geq3$. By Eq \ref{99} and Lemma \ref{ref6}(b), $\reg(\mathit{S}_{n,p,q}/(\mathit{I}_{n,p,q}:\mathit{x}_{n+1}))=\reg(\mathit{S}^\star_{n-2,p,q}/\mathit{I}^\star_{n-2,p,q}).$
		By Lemma \ref{Lem2}(b), $\reg(\mathit{S}_{n,p,q}/(\mathit{I}_{n,p,q}:\mathit{x}_{n+1}))=(n-2+1)p=(n-1)p$ and by Eq \ref{1010} and Lemma \ref{ref6}(b), $\reg(\mathit{S}_{n,p,q}/(\mathit{I}_{n,p,q},\mathit{x}_{n+1}))=\reg(\mathit{S}^\star_{n-1,p,q}/\mathit{I}^\star_{n-1,p,q}).$ Again by Lemma \ref{Lem2}(b), we have
$\reg(\mathit{S}_{n,p,q}/(\mathit{I}_{n,p,q},\mathit{x}_{n+1}))=(n-1+1)p=np.$
		Thus by Theorem \ref{reg}(c),
		$\reg(\mathit{S}_{n,p,q}/\mathit{I}_{n,p,q})=np.$ 

Now we prove the result for projective dimension by using Auslander–Buchsbaum formula $${\pdim}(S_{n,p,q}/\mathit{I}_{n,p,q})={\depth}(S_{n,p,q})-{\depth}(S_{n,p,q}/\mathit{I}_{n,p,q}).$$ Hence  ${\pdim}(S_{n,p,q}/\mathit{I}_{n,p,q})=(1+q)(1+n+np)-\bigl((p+q)n+q\bigr)=(1+pq)n+1.$
		\end{proof} 
	    \section{Depth, Stanley Depth, regularity and projective dimension of cyclic modules associated to $\mathcal{O}_n$, $\mathcal{O}_n(p)$, $\brs_q(\mathcal{O}_n)$ and $\brs_q(\mathcal{O}_n(p))$}\label{sec4}
		In this section we compute the exact values of regularity for the cyclic module $\mathit{Q}_{n,p}/\mathit{J}_{n,p}$. For this purpose we first compute the exact value of regularity for the cyclic module $\mathit{Q}^\star_{n,p}/\mathit{J}^\star_{n,p}.$ Shahid et al. in \cite{SMM} computed the exact values of depth and Stanley depth for these modules. The values of projective dimension for these modules can be found by using Theorem \ref{auss}. Furthermore, we compute the exact values of depth, Stanley depth, regularity and projective dimension for the cyclic module $\mathit{Q}_{n,p,q}/\mathit{J}_{n,p,q}$. For this purpose we first compute the exact values of all mentioned invariants for the cyclic module $\mathit{Q}^\star_{n,p,q}/\mathit{J}^\star_{n,p,q}$.
	\begin{Remark}\label{R2}
		{\em
		While proving our results by induction on $n$, sometimes we may have the description $Q^\star_{-2,p}/\mathit{J}^\star_{-2,p}$, $Q^\star_{-1,p}/\mathit{J}^\star_{-1,p}$, $Q^\star_{0,p}/\mathit{J}^\star_{0,p}$, $Q^\star_{-1,p,q}/\mathit{J}^\star_{-1,p,q}$ or $Q^\star_{0,p,q}/\mathit{J}^\star_{0,p,q}$, in that case we define
			\begin{itemize}
				\item $ Q^\star_{-2,p}/\mathit{J}^\star_{-2,p}\cong K$,
					\item
				$ Q^\star_{-1,p}/\mathit{J}^\star_{-1,p}\cong K[\mathit{B}_{1}]\cong K[\mathit{B}_{3}]\cong K[\mathit{B}_{4}]$,
				\item $ Q^\star_{0,p}/\mathit{J}^\star_{0,p}\cong K[V(\mathcal{S}_{2p})]/I(\mathcal{S}_{2p})$,
               \item $Q^\star_{-1,p,q}/\mathit{J}^\star_{-1,p,q}\cong\bigotimes_{k=1}^{p}{}_KK[V({\mathcal{S}_{q}})]/I(\mathcal{S}_{q}),$
				\item
				$ Q^\star_{0,p,q}/\mathit{J}^\star_{0,p,q}\cong K[V(\mathcal{S}_{2p,q})]/I(\mathcal{S}_{2p,q})$.
			\end{itemize} 
		 } 
	\end{Remark}
	\begin{Lemma}\label{Lem3}
 {\em
		Let $n,p\geq1$. Then $\reg(Q^\star_{n,p}/\mathit{J}^\star_{n,p})=\left\lceil{\frac{n+1}{2}}\right\rceil$.
  }
	\end{Lemma}
	\begin{proof}
We will prove the result by induction on $n$. We have the following isomorphisms:
  \begin{equation}\label{444}
Q^\star_{n,p}/(\mathit{J}^\star_{n,p}:\mathit{x}_{n})\cong Q^\star_{n-3,p}/\mathit{J}^\star_{n-3,p}\bigotimes_KK[\{\mathit{x}_{n}\}\cup\mathit{B}_{n+1}],
\end{equation} and
\begin{equation}\label{555}
Q^\star_{n,p}/(\mathit{J}^\star_{n,p},\mathit{x}_{n})\cong Q^\star_{n-2,p}/\mathit{J}^\star_{n-2,p}\bigotimes_KK[V(\mathcal{S}_{2p})]/I(\mathcal{S}_{2p}).
\end{equation} If $n=1$, then Eq \ref{444} becomes $Q^\star_{1,p}/(\mathit{J}^\star_{1,p}:\mathit{x}_{1})\cong Q^\star_{-2,p}/\mathit{J}^\star_{-2,p}\bigotimes_KK[\{\mathit{x}_{1}\}\cup\mathit{B}_{2}]$ and by Remark \ref{R2}, $Q^\star_{1,p}/(\mathit{J}^\star_{1,p}:\mathit{x}_{1})\cong K\bigotimes_KK[\{\mathit{x}_{1}\}\cup\mathit{B}_{2}]\cong K[\{\mathit{x}_{1}\}\cup\mathit{B}_{2}]$, which implies $\reg(Q^\star_{1,p}/(\mathit{J}^\star_{1,p}:\mathit{x}_{1}))=0.$ Now by Eq \ref{555}, $Q^\star_{1,p}/(\mathit{J}^\star_{1,p},\mathit{x}_{1})\cong Q^\star_{-1,p}/\mathit{J}^\star_{-1,p}\bigotimes_KK[V(\mathcal{S}_{2p})]/I(\mathcal{S}_{2p})$ and again by Remark \ref{R2}, we get $Q^\star_{1,p}/(\mathit{J}^\star_{1,p},\mathit{x}_{1})\cong K[\mathit{B}_{3}]\bigotimes_KK[V(\mathcal{S}_{2p})]/I(\mathcal{S}_{2p}).$ Using Lemma \ref{ref6}(b) $$\reg(Q^\star_{1,p}/(\mathit{J}^\star_{1,p},\mathit{x}_{1}))=\reg(K[V(\mathcal{S}_{2p})]/I(\mathcal{S}_{2p})).$$ By Theorem \ref{star}(b), $\reg(Q^\star_{1,p}/(\mathit{J}^\star_{1,p},\mathit{x}_{1}))=1.$ The result follows by Theorem \ref{reg}(c), that is
\begin{equation}\label{888}
\reg(Q^\star_{1,p}/\mathit{J}^\star_{1,p})=1=\left\lceil{\frac{1+1}{2}}\right\rceil.
\end{equation}
If $n=2$, then by Eq \ref{444}, $Q^\star_{2,p}/(\mathit{J}^\star_{2,p}:\mathit{x}_{2})\cong Q^\star_{-1,p}/\mathit{J}^\star_{-1,p}\bigotimes_KK[\{\mathit{x}_{2}\}\cup\mathit{B}_{3}]$ and by Remark \ref{R2}, $Q^\star_{2,p}/(\mathit{J}^\star_{2,p}:\mathit{x}_{2})\cong K[\mathit{B}_{4}]\bigotimes_KK[\{\mathit{x}_{2}\}\cup\mathit{B}_{3}]\cong K[\{\mathit{x}_{2}\}\cup\mathit{B}_{3}\cup\mathit{B}_{4}],$ we have, $\reg(Q^\star_{2,p}/(\mathit{J}^\star_{2,p}:\mathit{x}_{2}))=0.$ By Eq \ref{555}, $Q^\star_{2,p}/(\mathit{J}^\star_{2,p},\mathit{x}_{2})\cong Q^\star_{0,p}/\mathit{J}^\star_{0,p}\bigotimes_KK[V(\mathcal{S}_{2p})]/I(\mathcal{S}_{2p}),$ which by Remark \ref{R2} gives$$Q^\star_{2,p}/(\mathit{J}^\star_{2,p},\mathit{x}_{2})\cong K[V(\mathcal{S}_{2p})]/I(\mathcal{S}_{2p})\bigotimes_KK[V(\mathcal{S}_{2p})]/I(\mathcal{S}_{2p}).$$ By Lemma \ref{circulentt}, $\reg(Q^\star_{2,p}/(\mathit{J}^\star_{2,p},\mathit{x}_{2}))=\reg(K[V(\mathcal{S}_{2p})]/I(\mathcal{S}_{2p}))+\reg(K[V(\mathcal{S}_{2p})]/I(\mathcal{S}_{2p})).$ By Theorem \ref{star}(b), $\reg(Q^\star_{2,p}/(\mathit{J}^\star_{2,p},\mathit{x}_{2}))=1+1=2$. By Theorem \ref{reg}(c), we have
\begin{equation}\label{999}
\reg(Q^\star_{2,p}/\mathit{J}^\star_{2,p})=2=\left\lceil{\frac{2+1}{2}}\right\rceil.
\end{equation}
Now let $n\geq3$, the result follows by applying induction on $n$. By Eq \ref{444} and Lemma \ref{ref6}(b), $\reg(Q^\star_{n,p}/(\mathit{J}^\star_{n,p}:\mathit{x}_{n}))=\reg(Q^\star_{n-3,p}/\mathit{J}^\star_{n-3,p}).$
		By induction, $\reg(Q^\star_{n,p}/(\mathit{J}^\star_{n,p}:\mathit{x}_{n}))=\left\lceil{\frac{n-3+1}{2}}\right\rceil=\left\lceil{\frac{n-2}{2}}\right\rceil.$
		Using Lemma \ref{circulentt} on Eq \ref{555} 
$$\reg\big(Q^\star_{n,p}/(\mathit{J}^\star_{n,p},\mathit{x}_{n})\big)= \reg(Q^\star_{n-2,p}/\mathit{J}^\star_{n-2,p})+\reg(K[V({\mathcal{S}_{2p}})]/I(\mathcal{S}_{2p}).$$
		Again by induction and Theorem \ref{star}(b),
$\reg(Q^\star_{n,p}/(\mathit{J}^\star_{n,p},\mathit{x}_{n}))=\left\lceil{\frac{n-2+1}{2}}\right\rceil+1=\left\lceil{\frac{n+1}{2}}\right\rceil$. Hence, by Theorem \ref{reg}(c),
	$	\reg(Q^\star_{n,p}/\mathit{J}^\star_{n,p})=\left\lceil{\frac{n+1}{2}}\right\rceil.$

		
	\end{proof}
	\begin{Theorem}\label{Th3}
 {\em
		Let $n,p\geq3$. Then $\reg(Q_{n,p}/\mathit{J}_{n,p})=\left\lceil{\frac{n-1}{2}}\right\rceil$.
  }
	\end{Theorem}
	\begin{proof}
		We have the following isomorphisms:
		\begin{equation}\label{666}
Q_{n,p}/(\mathit{J}_{n,p}:\mathit{x}_{n})\cong Q^\star_{n-4,p}/\mathit{J}^\star_{n-4,p}\bigotimes_KK[\mathit{x}_{n}],
\end{equation}and
\begin{equation}\label{777}
Q_{n,p}/(\mathit{J}_{n,p},\mathit{x}_{n})\cong Q^\star_{n-2,p}/\mathit{J}^\star_{n-2,p}.
\end{equation} 
If $n=3$, then Eq \ref{666} becomes, $Q_{3,p}/(\mathit{J}_{3,p}:\mathit{x}_{3})\cong Q^\star_{-1,p}/\mathit{J}^\star_{-1,p}\bigotimes_KK[\mathit{x}_{3}].$ By Remark \ref{R2}, $Q_{3,p}/(\mathit{J}_{3,p}:\mathit{x}_{3})\cong K[\mathit{B}_{1}]\bigotimes_KK[\mathit{x}_{3}]\cong K[\{\mathit{x}_{3}\}\cup\mathit{B}_{1}],$ thus $\reg(Q_{3,p}/(\mathit{J}_{3,p}:\mathit{x}_{3}))=0$ and Eq \ref{777} becomes, $Q_{3,p}/(\mathit{J}_{3,p},\mathit{x}_{3})\cong Q^\star_{1,p}/\mathit{J}^\star_{1,p}.$ Using Eq \ref{888}, we have, $\reg(Q_{3,p}/(\mathit{J}_{3,p},\mathit{x}_{3}))=1.$ Now by Theorem \ref{reg}(c), we get $\reg(Q_{3,p}/\mathit{J}_{3,p}))=1=\left\lceil{\frac{3-1}{2}}\right\rceil.$\\ If $n=4$, then Eq \ref{666} implies that $Q_{4,p}/(\mathit{J}_{4,p}:\mathit{x}_{4})\cong Q^\star_{0,p}/\mathit{J}^\star_{0,p}\bigotimes_KK[\mathit{x}_{4}].$ Using Remark \ref{R2}, $Q_{4,p}/(\mathit{J}_{4,p}:\mathit{x}_{4})\cong K[V(\mathcal{S}_{2p})]/I(\mathcal{S}_{2p})\bigotimes_KK[\mathit{x}_{4}].$ By Lemma \ref{ref6}(b), $\reg(Q_{4,p}/(\mathit{J}_{4,p}:\mathit{x}_{4}))=\reg(K[V(\mathcal{S}_{2p})]/I(\mathcal{S}_{2p})).$ By Theorem \ref{star}(b), $\reg(Q_{4,p}/(\mathit{J}_{4,p}:\mathit{x}_{4}))=1.$ By Eq \ref{777}, $Q_{4,p}/(\mathit{J}_{4,p},\mathit{x}_{4})\cong Q^\star_{2,p}/\mathit{J}^\star_{2,p}.$ By Lemma  \ref{Lem3}, $\reg(Q_{4,p}/(\mathit{J}_{4,p},\mathit{x}_{4}))= 2.$ By Theorem \ref{reg}(c), $\reg(Q_{4,p}/\mathit{J}_{4,p})= 2=\left\lceil{\frac{4-1}{2}}\right\rceil.$\\ Now let $n\geq5$. It follow from Eq \ref{666} and Lemma \ref{ref6}(b) that $$\reg(Q_{n,p}/(\mathit{J}_{n,p}:\mathit{x}_{n}))=\reg(Q^\star_{n-4,p}/\mathit{J}^\star_{n-4,p}).$$ From Lemma \ref{Lem3} we have
$\reg(Q_{n,p}/(\mathit{J}_{n,p}:\mathit{x}_{n}))=\left\lceil{\frac{n-4+1}{2}}\right\rceil=\left\lceil{\frac{n-3}{2}}\right\rceil$ and using Eq \ref{777} $$\reg(Q_{n,p}/(\mathit{J}_{n,p},\mathit{x}_{n})))=\reg(Q^\star_{n-2,p}/\mathit{J}^\star_{n-2,p}).$$
		Again by Lemma \ref{Lem3}
		$\reg(Q_{n,p}/(\mathit{J}_{n,p},\mathit{x}_{n}))=\left\lceil{\frac{n-2+1}{2}}\right\rceil=\left\lceil{\frac{n-1}{2}}\right\rceil.$
		Hence by Theorem \ref{reg}(c) we have
$\reg(Q_{n,p}/\mathit{J}_{n,p})=\left\lceil{\frac{n-1}{2}}\right\rceil.$
	\end{proof}
	\begin{Lemma}\label{Lem4}
 {\em
		Let $n,p,q\geq 1$. Then
  }
		\begin{itemize}
			\item[(a)] $\depth(\mathit{Q}^\star_{n,p,q}/\mathit{J}^\star_{n,p,q})=\sdepth(\mathit{Q}^\star_{n,p,q}/\mathit{J}^\star_{n,p,q})=(p+q)(n+1)+p,$
			\item[(b)]$\reg(Q^\star_{n,p,q}/\mathit{J}^\star_{n,p,q})=(n+2)p,$
			\item[(c)]$\pdim(Q^\star_{n,p,q}/\mathit{J}^\star_{n,p,q})=(1+pq)(n+1)+pq$.
		\end{itemize}  
	\end{Lemma}
	\begin{proof}
		First we prove the result for depth. Consider the short exact sequence
		\begin{equation}\label{es4}
		0\longrightarrow \mathit{Q}^\star_{n,p,q}/(\mathit{J}^\star_{n,p,q}:\mathit{x}_{n+1})\xrightarrow{\cdot \mathit{x}_{n+1}} \mathit{Q}^\star_{n,p,q}/\mathit{J}^\star_{n,p,q}\longrightarrow \mathit{Q}^\star_{n,p,q}/(\mathit{J}^\star_{n,p,q},\mathit{x}_{n+1})\longrightarrow 0,
		\end{equation}
		by applying Depth Lemma and Lemma \ref{sdepth} on Eq \ref{es4}, we get
		\begin{equation}\label{f}
		\depth(\mathit{Q}^\star_{n,p,q}/\mathit{J}^\star_{n,p,q})\geq \min\{\depth(\mathit{Q}^\star_{n,p,q}/(\mathit{J}^\star_{n,p,q}:\mathit{x}_{n+1})), \depth(\mathit{Q}^\star_{n,p,q}/(\mathit{J}^\star_{n,p,q},\mathit{x}_{n+1}))\},
		\end{equation}
        and
        \begin{equation}\label{ssf}
		\sdepth(\mathit{Q}^\star_{n,p,q}/\mathit{J}^\star_{n,p,q})\geq \min\{\sdepth(\mathit{Q}^\star_{n,p,q}/(\mathit{J}^\star_{n,p,q}:\mathit{x}_{n+1})), \sdepth(\mathit{Q}^\star_{n,p,q}/(\mathit{J}^\star_{n,p,q},\mathit{x}_{n+1}))\}.
		\end{equation}
		We have the following isomorphisms:
\begin{equation}\label{es5*}
\mathit{Q}^\star_{n,p,q}/(\mathit{J}^\star_{n,p,q}:\mathit{x}_{n+1})\cong \mathit{Q}^\star_{n-2,p,q}/\mathit{J}^\star_{n-2,p,q}\bigotimes_KK[\{\mathit{x}_{n+1}\}\cup\mathit{A}_{n}\cup\mathit{C}_{n}\cup\mathit{C}_{n+1}],
		\end{equation}
		\begin{equation}\label{es6*}
		\mathit{Q}^\star_{n,p,q}/(\mathit{J}^\star_{n,p,q},\mathit{x}_{n+1})\cong \mathit{Q}^\star_{n-1,p,q}/\mathit{J}^\star_{n-1,p,q}\bigotimes_KK[\mathit{A}_{n+1}]\bigotimes_{k=1}^{p}{}_KK[V(\mathcal{S}_{q})]/{I(\mathcal{S}_{q})},
		\end{equation} and 
  \begin{equation}\label{999}
\mathit{Q}^\star_{n,p,q}/(\mathit{J}^\star_{n,p,q}:\mathit{y}_{(n+1)1}\mathit{y}_{(n+1)2}\dots\mathit{y}_{(n+1)p})\cong \mathit{Q}^\star_{n-1,p,q}/\mathit{J}^\star_{n-1,p,q}\bigotimes_KK[\mathit{A}_{n+1}\cup\mathit{B}_{n+1}].
		\end{equation}
		If $n=1$, then Eq \ref{es5*} becomes, $\mathit{Q}^\star_{1,p,q}/(\mathit{J}^\star_{1,p,q}:\mathit{x}_{2})\cong \mathit{Q}^\star_{-1,p,q}/\mathit{J}^\star_{-1,p,q}\bigotimes_KK[\{\mathit{x}_{2}\}\cup\mathit{A}_{1}\cup\mathit{C}_{1}\cup\mathit{C}_{2}].$ By Remark \ref{R2}
  \begin{equation}\label{aa}
\mathit{Q}^\star_{1,p,q}/(\mathit{J}^\star_{1,p,q}:\mathit{x}_{2})\cong \bigotimes_{k=1}^{p} K[V({\mathcal{S}_{q}})]/I(\mathcal{S}_{q})\bigotimes_KK[\{\mathit{x}_{2}\}\cup\mathit{A}_{1}\cup\mathit{C}_{1}\cup\mathit{C}_{2}].
\end{equation}
By Lemma \ref{ref6}(a) $$\depth(\mathit{Q}^\star_{1,p,q}/(\mathit{J}^\star_{1,p,q}:\mathit{x}_{2}))=\sum_{k=1}^{p}\depth(K[V({\mathcal{S}_{q}})]/I(\mathcal{S}_{q}))+\depth(K[\{\mathit{x}_{2}\}\cup\mathit{A}_{1}\cup\mathit{C}_{1}\cup\mathit{C}_{2}]),$$
and 
$$\sdepth(\mathit{Q}^\star_{1,p,q}/(\mathit{J}^\star_{1,p,q}:\mathit{x}_{2}))=\sum_{k=1}^{p}\sdepth(K[V({\mathcal{S}_{q}})]/I(\mathcal{S}_{q}))+\sdepth(K[\{\mathit{x}_{2}\}\cup\mathit{A}_{1}\cup\mathit{C}_{1}\cup\mathit{C}_{2}]).$$ 
Using Theorem \ref{star}(a), $\depth(\mathit{Q}^\star_{1,p,q}/(\mathit{J}^\star_{1,p,q}:\mathit{x}_{2}))=\sdepth(\mathit{Q}^\star_{1,p,q}/(\mathit{J}^\star_{1,p,q}:\mathit{x}_{2}))=\sum_{k=1}^{p}1+1+q+pq+pq=1+p+q+2pq.$ Also by Eq \ref{es6*}, $$\mathit{Q}^\star_{1,p,q}/(\mathit{J}^\star_{1,p,q},\mathit{x}_{2})\cong \mathit{Q}^\star_{0,p,q}/\mathit{J}^\star_{0,p,q}\bigotimes_K K[\mathit{A}_{2}]\bigotimes_{k=1}^{p} K[V(\mathcal{S}_{q})]/I(\mathcal{S}_{q}),$$ 
again by Remark \ref{R2}
\begin{equation}\label{bb}
\mathit{Q}^\star_{1,p,q}/(\mathit{J}^\star_{1,p,q},\mathit{x}_{2})\cong K[V(\mathcal{S}_{2p,q})]/I(\mathcal{S}_{2p,q})\bigotimes_K K[\mathit{A}_{2}]\bigotimes_{k=1}^{p} K[V(\mathcal{S}_{q})]/{I(\mathcal{S}_{q})}.
\end{equation}
By Lemma \ref{ref3} and Lemma \ref{ref5}, it follows that
\begin{align*}
\depth(\mathit{Q}^\star_{1,p,q}/(\mathit{J}^\star_{1,p,q},\mathit{x}_{2}))=& \depth(K[V(\mathcal{S}_{2p,q})]/I(\mathcal{S}_{2p,q}))+\depth(K[\mathit{A}_{2}])\\
&+\sum_{k=1}^{p}\depth(K[V(\mathcal{S}_{q})]/{I(\mathcal{S}_{q})}).
\end{align*}
and
\begin{align*}
\sdepth(\mathit{Q}^\star_{1,p,q}/(\mathit{J}^\star_{1,p,q},\mathit{x}_{2}))\geq& \sdepth(K[V(\mathcal{S}_{2p,q})]/I(\mathcal{S}_{2p,q}))+\sdepth(K[\mathit{A}_{2}])\\
&+\sum_{k=1}^{p}\sdepth(K[V(\mathcal{S}_{q})]/{I(\mathcal{S}_{q})}).
\end{align*}
Using Corollary \ref{bstar}(a) and Theorem \ref{star}(a) $$\depth(\mathit{Q}^\star_{1,p,q}/(\mathit{J}^\star_{1,p,q},\mathit{x}_{2}))=2p+q+q+\sum_{k=1}^{p}1=2p+2q+p=3p+2q,$$
and
$$\sdepth(\mathit{Q}^\star_{1,p,q}/(\mathit{J}^\star_{1,p,q},\mathit{x}_{2}))\geq 2p+q+q+\sum_{k=1}^{p}1=2p+2q+p=3p+2q.$$Thus by Eq \ref{f} and Eq \ref{ssf}, $\depth(\mathit{Q}^\star_{1,p,q}/\mathit{J}^\star_{1,p,q}),\sdepth(\mathit{Q}^\star_{1,p,q}/\mathit{J}^\star_{1,p,q})\geq3p+2q.$ Now since $\mathit{y}_{21}\mathit{y}_{22}\dots\mathit{y}_{2p}\notin\mathit{J}^\star_{1,p,q},$ so by Eq \ref{999}, $\mathit{Q}^\star_{1,p,q}/(\mathit{J}^\star_{1,p,q}:\mathit{y}_{21}\mathit{y}_{22}\dots\mathit{y}_{2p})\cong\mathit{Q}^\star_{0,p,q}/\mathit{J}^\star_{0,p,q}\bigotimes_KK[\mathit{A}_{2}\cup\mathit{B}_{2}],$ which by Remark \ref{R2} implies that $\mathit{Q}^\star_{1,p,q}/(\mathit{J}^\star_{1,p,q}:\mathit{y}_{21}\mathit{y}_{22}\dots\mathit{y}_{2p})\cong K[V(\mathcal{S}_{2p,q})]/I(\mathcal{S}_{2p,q})\bigotimes_KK[\mathit{A}_{2}\cup\mathit{B}_{2}].$ Applying Lemma \ref{ref6}(a), we have $$\depth(\mathit{Q}^\star_{1,p,q}/(\mathit{J}^\star_{1,p,q}:\mathit{y}_{21}\mathit{y}_{22}\dots\mathit{y}_{2p}))=\depth(K[V(\mathcal{S}_{2p,q})]/I(\mathcal{S}_{2p,q}))+\depth(K[\mathit{A}_{2}\cup\mathit{B}_{2}]),$$
and
$$\sdepth(\mathit{Q}^\star_{1,p,q}/(\mathit{J}^\star_{1,p,q}:\mathit{y}_{21}\mathit{y}_{22}\dots\mathit{y}_{2p}))=\sdepth(K[V(\mathcal{S}_{2p,q})]/I(\mathcal{S}_{2p,q}))+\sdepth(K[\mathit{A}_{2}\cup\mathit{B}_{2}]).$$ Using Corollary \ref{bstar}(a), $$\depth(\mathit{Q}^\star_{1,p,q}/(\mathit{J}^\star_{1,p,q}:\mathit{y}_{21}\mathit{y}_{22}\dots\mathit{y}_{2p}))=\sdepth(\mathit{Q}^\star_{1,p,q}/(\mathit{J}^\star_{1,p,q}:\mathit{y}_{21}\mathit{y}_{22}\dots\mathit{y}_{2p}))=2p+q+q+p=3p+2q.$$ Thus by Corollary \ref{c1} and Proposition \ref{c2}, $$\depth(Q^\star_{1,p,q}/\mathit{J}^\star_{1,p,q}), \sdepth(Q^\star_{1,p,q}/\mathit{J}^\star_{1,p,q})\leq3p+2q.$$ Hence
\begin{equation}\label{ee}
\depth(Q^\star_{1,p,q}/\mathit{J}^\star_{1,p,q})=\sdepth(Q^\star_{1,p,q}/\mathit{J}^\star_{1,p,q})=3p+2q.
\end{equation}
If $n=2$, then by using the similar arguments and case $n=1$, one can easily verify the required result, that is, $\depth(Q^\star_{2,p,q}/\mathit{J}^\star_{2,p,q})=\sdepth(Q^\star_{2,p,q}/\mathit{J}^\star_{2,p,q})=4p+3q.$
		Let $n\geq3$, the result follows by applying induction on $n$.
		Using Eq \ref{es5*} and Lemma \ref{ref6}(a), we have$$\depth(\mathit{Q}^\star_{n,p,q}/(\mathit{J}^\star_{n,p,q}:\mathit{x}_{n+1}))=\depth(\mathit{Q}^\star_{n-2,p,q}/\mathit{J}^\star_{n-2,p,q})+\depth(K[\{\mathit{x}_{n+1}\}\cup\mathit{A}_{n}\cup\mathit{C}_{n}\cup\mathit{C}_{n+1}]),$$
        and        $$\sdepth(\mathit{Q}^\star_{n,p,q}/(\mathit{J}^\star_{n,p,q}:\mathit{x}_{n+1}))=\sdepth(\mathit{Q}^\star_{n-2,p,q}/\mathit{J}^\star_{n-2,p,q})+\sdepth(K[\{\mathit{x}_{n+1}\}\cup\mathit{A}_{n}\cup\mathit{C}_{n}\cup\mathit{C}_{n+1}]),$$
        by induction $$\depth(\mathit{Q}^\star_{n,p,q}/(\mathit{J}^\star_{n,p,q}:\mathit{x}_{n+1}))=(p+q)(n-2+1)+1+q+pq+pq=(p+q)n+2pq+1,$$
        and
        $$\sdepth(\mathit{Q}^\star_{n,p,q}/(\mathit{J}^\star_{n,p,q}:\mathit{x}_{n+1}))=(p+q)(n-2+1)+1+q+pq+pq=(p+q)n+2pq+1.$$
        Now by Eq \ref{es6*}, Lemma \ref{ref3} and Lemma \ref{ref5}, we have  
		\begin{align*} \depth\big(	\mathit{Q}^\star_{n,p,q}/(\mathit{J}^\star_{n,p,q},\mathit{x}_{n+1})\big)=&\depth(\mathit{Q}^\star_{n-1,p,q}/\mathit{J}^\star_{n-1,p,q})+\depth(K[\mathit{A}_{n+1}])\\
  &+\sum_{j=1}^{p}\depth(K[V(\mathcal{S}_{q})]/{I(\mathcal{S}_{q})}),
  \end{align*}
  and
  \begin{align*} \sdepth\big(	\mathit{Q}^\star_{n,p,q}/(\mathit{J}^\star_{n,p,q},\mathit{x}_{n+1})\big)\geq&\sdepth(\mathit{Q}^\star_{n-1,p,q}/\mathit{J}^\star_{n-1,p,q})+\sdepth(K[\mathit{A}_{n+1}])\\
  &+\sum_{j=1}^{p}\sdepth(K[V(\mathcal{S}_{q})]/{I(\mathcal{S}_{q})}),
  \end{align*}
	by induction and Theorem \ref{star}(a)
\begin{multline*}
\depth(\mathit{Q}^\star_{n,p,q}/(\mathit{J}^\star_{n,p,q},\mathit{x}_{n+1}))=(p+q)(n-1+1)+p+q+\sum_{j=1}^{p}1=(p+q)n+p+q+p=(p+q)(n+1)+p,\end{multline*}
and
\begin{multline*}
\sdepth(\mathit{Q}^\star_{n,p,q}/(\mathit{J}^\star_{n,p,q},\mathit{x}_{n+1}))\geq(p+q)(n-1+1)+p+q+\sum_{j=1}^{p}1=(p+q)n+p+q+p=(p+q)(n+1)+p.\end{multline*}
Again using Eq \ref{f}, we get $\depth(\mathit{Q}^\star_{n,p,q}/\mathit{J}^\star_{n,p,q}),\sdepth(\mathit{Q}^\star_{n,p,q}/\mathit{J}^\star_{n,p,q})\geq (p+q)(n+1)+p.$ Now since $\mathit{y}_{(n+1)1}\mathit{y}_{(n+1)2}\dots\mathit{y}_{(n+1)p}\notin\mathit{J}^\star_{n,p,q},$ so by Eq \ref{999} and Lemma \ref{ref6}(a) \begin{align*}
\depth(\mathit{Q}^\star_{n,p,q}/(\mathit{J}^\star_{n,p,q}:\mathit{y}_{(n+1)1}\mathit{y}_{(n+1)2}\dots\mathit{y}_{(n+1)p}))=&\depth(\mathit{Q}^\star_{n-1,p,q}/\mathit{J}^\star_{n-1,p,q})\\
&+\depth(K[\mathit{A}_{n+1}\cup\mathit{B}_{n+1}]),\end{align*}
and
\begin{align*}
\sdepth(\mathit{Q}^\star_{n,p,q}/(\mathit{J}^\star_{n,p,q}:\mathit{y}_{(n+1)1}\mathit{y}_{(n+1)2}\dots\mathit{y}_{(n+1)p}))=&\sdepth(\mathit{Q}^\star_{n-1,p,q}/\mathit{J}^\star_{n-1,p,q})\\
&+\sdepth(K[\mathit{A}_{n+1}\cup\mathit{B}_{n+1}]).\end{align*}
		 By induction $$\depth(\mathit{Q}^\star_{n,p,q}/(\mathit{J}^\star_{n,p,q}:\mathit{y}_{(n+1)1}\mathit{y}_{(n+1)2}\dots\mathit{y}_{(n+1)p}))=(p+q)(n-1+1)+p+p+q=(p+q)(n+1)+p,$$
         and        $$\sdepth(\mathit{Q}^\star_{n,p,q}/(\mathit{J}^\star_{n,p,q}:\mathit{y}_{(n+1)1}\mathit{y}_{(n+1)2}\dots\mathit{y}_{(n+1)p}))=(p+q)(n-1+1)+p+p+q=(p+q)(n+1)+p.$$
		Thus by Corollary \ref{c1} and Proposition \ref{c2}, $\depth(\mathit{Q}^\star_{n,p,q}/\mathit{J}^\star_{n,p,q}),\sdepth(\mathit{Q}^\star_{n,p,q}/\mathit{J}^\star_{n,p,q})\leq (p+q)(n+1)+p.$ Hence $\depth(\mathit{Q}^\star_{n,p,q}/\mathit{J}^\star_{n,p,q})=\sdepth(\mathit{Q}^\star_{n,p,q}/\mathit{J}^\star_{n,p,q})=(p+q)(n+1)+p.$\\
Now we prove the result for regularity. If $n=1,$ then by using Eq \ref{aa} and Lemma \ref{circulentt}$$\reg(\mathit{Q}^\star_{1,p,q}/(\mathit{J}^\star_{1,p,q}:\mathit{x}_{2}))= \sum_{k=1}^{p}\reg(K[V({\mathcal{S}_{q}})]/I(\mathcal{S}_{q})).$$ By Theorem \ref{star}(b), $\reg(\mathit{Q}^\star_{1,p,q}/(\mathit{J}^\star_{1,p,q}:\mathit{x}_{2}))= \sum_{k=1}^{p}1=p.$ Applying Lemma \ref{ref6}(b) and Lemma \ref{circulentt} on Eq \ref{bb}, we have $$\reg(\mathit{Q}^\star_{1,p,q}/(\mathit{J}^\star_{1,p,q},\mathit{x}_{2}))=\reg(K[V(\mathcal{S}_{2p,q})]/I(\mathcal{S}_{2p,q}))+\sum_{k=1}^{p}\reg(K[V(\mathcal{S}_{q})]/{I(\mathcal{S}_{q})}).$$ Using Corollary \ref{bstar}(b) and Theorem \ref{star}(b),
$\reg(\mathit{Q}^\star_{1,p,q}/(\mathit{J}^\star_{1,p,q},\mathit{x}_{2}))=2p+\sum_{k=1}^{p}1=2p+p=3p.$
By Theorem \ref{reg}(c)
\begin{equation}\label{hh}
\reg(\mathit{Q}^\star_{1,p,q}/\mathit{J}^\star_{1,p,q})=3p.
\end{equation} If $n=2$, then by using similar arguments and case $n=1$, the desired result can be easily verified, that is $\reg\big(Q^\star_{2,p,q}/\mathit{J}^\star_{2,p,q}\big)=4p.$ Now let $n\geq3$, the result follows by applying induction on $n$. By Eq \ref{es5*} and Lemma \ref{ref6}(b), we have $$\reg(Q^\star_{n,p,q}/(\mathit{J}^\star_{n,p,q}:\mathit{x}_{n+1}))= \reg(\mathit{Q}^\star_{n-2,p,q}/\mathit{J}^\star_{n-2,p,q}).$$ By induction, $\reg(Q^\star_{n,p,q}/(\mathit{J}^\star_{n,p,q}:\mathit{x}_{n+1}))=(n-2+2)p=np.$ Using Lemma \ref{circulentt} and  Lemma \ref{ref6}(b) on Eq \ref{es6*}, $$\reg(\mathit{Q}^\star_{n,p,q}/(\mathit{J}^\star_{n,p,q},\mathit{x}_{n+1}))=\reg(\mathit{Q}^\star_{n-1,p,q}/\mathit{J}^\star_{n-1,p,q})+\sum_{k=1}^{p}\reg(K[V(\mathcal{S}_{q})]/{I(\mathcal{S}_{q})}).$$ By induction and Theorem \ref{star}(b),
$\reg(Q^\star_{n,p,q}/(\mathit{J}^\star_{n,p,q},\mathit{x}_{n+1}))=(n-1+2)p+\sum_{k=1}^{p}1=(n+1)p+p=(n+2)p.$ Thus by Theorem \ref{reg}(c), $\reg(Q^\star_{n,p,q}/\mathit{J}^\star_{n,p,q})=(n+2)p.$ The result for projective dimension can easily be proved by using Auslander–Buchsbaum formula.

		
	\end{proof}
	\begin{Theorem}\label{Th4}
 {\em
		Let $p,q\geq 1$ and $n\geq3$. Then
		\begin{itemize}
			\item[(a)] $\depth(\mathit{Q}_{n,p,q}/\mathit{J}_{n,p,q})=\sdepth(\mathit{Q}_{n,p,q}/\mathit{J}_{n,p,q})=n(p+q),$
			\item[(b)]$\reg(\mathit{Q}_{n,p,q}/\mathit{J}_{n,p,q})=np,$
			\item[(c)]$\pdim(\mathit{Q}_{n,p,q}/\mathit{J}_{n,p,q})=(1+pq)n$.
		\end{itemize}
  }
	\end{Theorem}
	\begin{proof}
		First we prove the result for depth. Consider the short exact sequence.
		\begin{equation}\label{es7}
		0\longrightarrow \mathit{Q}_{n,p,q}/(\mathit{J}_{n,p,q}:\mathit{x}_{n})\xrightarrow{\cdot \mathit{x}_{n}} \mathit{Q}_{n,p,q}/\mathit{J}_{n,p,q}\longrightarrow \mathit{Q}_{n,p,q}/(\mathit{J}_{n,p,q},\mathit{x}_{n})\longrightarrow 0,
		\end{equation}
		by applying Depth Lemma and Lemma \ref{sdepth} on Eq \ref{es7}, we get
		\begin{equation}\label{g}
		\depth(\mathit{Q}_{n,p,q}/\mathit{J}_{n,p,q})\geq \min\{\depth(\mathit{Q}_{n,p,q}/(\mathit{J}_{n,p,q}:\mathit{x}_{n})), \depth(\mathit{Q}_{n,p,q}/(\mathit{J}_{n,p,q},\mathit{x}_{n}))\}.
		\end{equation}
        and
        \begin{equation}\label{ssg}
		\sdepth(\mathit{Q}_{n,p,q}/\mathit{J}_{n,p,q})\geq \min\{\sdepth(\mathit{Q}_{n,p,q}/(\mathit{J}_{n,p,q}:\mathit{x}_{n})), \sdepth(\mathit{Q}_{n,p,q}/(\mathit{J}_{n,p,q},\mathit{x}_{n}))\}.
		\end{equation}
		After renumbering the variables, we have the following isomorphisms:
		\begin{equation}\label{es7*}
		\mathit{Q}_{n,p,q}/(\mathit{J}_{n,p,q}:\mathit{x}_{n})\cong \mathit{Q}^\star_{n-4,p,q}/\mathit{J}^\star_{n-4,p,q}\bigotimes_KK[\{\mathit{x}_{n}\}\cup\mathit{A}_{1}\cup\mathit{A}_{n-1}\cup\mathit{C}_{n-1}\cup\mathit{C}_{n}],
		\end{equation}
		\begin{equation}\label{es8*}
		\mathit{Q}_{n,p,q}/(\mathit{J}_{n,p,q},\mathit{x}_{n})\cong \mathit{Q}^\star_{n-2,p,q}/\mathit{J}^\star_{n-2,p,q}\bigotimes_KK[\mathit{A}_{n}],
		\end{equation} and 
  \begin{equation}\label{cc}
\mathit{Q}_{n,p,q}/(\mathit{J}_{n,p,q}:\mathit{x}_{n1}\mathit{x}_{n2}\dots\mathit{x}_{nq})\cong \mathit{Q}^\star_{n-2,p,q}/\mathit{J}^\star_{n-2,p,q}\bigotimes_KK[\mathit{A}_{n}].
  \end{equation}
  If $n=3,$ then by Eq \ref{es7*}, $\mathit{Q}_{3,p,q}/(\mathit{J}_{3,p,q}:\mathit{x}_{3})\cong \mathit{Q}^\star_{-1,p,q}/\mathit{J}^\star_{-1,p,q}\bigotimes_KK[\{\mathit{x}_{3}\}\cup\mathit{A}_{1}\cup\mathit{A}_{2}\cup\mathit{C}_{2}\cup\mathit{C}_{3}].$ By Remark \ref{R2}
  \begin{equation}\label{dd}
\mathit{Q}_{3,p,q}/(\mathit{J}_{3,p,q}:\mathit{x}_{3})\cong \bigotimes_{k=1}^{p} K[V({\mathcal{S}_{q}})]/I(\mathcal{S}_{q})\bigotimes_KK[\{\mathit{x}_{3}\}\cup\mathit{A}_{1}\cup\mathit{A}_{2}\cup\mathit{C}_{2}\cup\mathit{C}_{3}],  
  \end{equation} applying Lemma \ref{ref6}(a)
  $$\depth(\mathit{Q}_{3,p,q}/(\mathit{J}_{3,p,q}:\mathit{x}_{3}))=\sum_{k=1}^{p}\depth(K[V({\mathcal{S}_{q}})]/I(\mathcal{S}_{q}))+\depth(K[\{\mathit{x}_{3}\}\cup\mathit{A}_{1}\cup\mathit{A}_{2}\cup\mathit{C}_{2}\cup\mathit{C}_{3}]),$$
  and  $$\sdepth(\mathit{Q}_{3,p,q}/(\mathit{J}_{3,p,q}:\mathit{x}_{3}))=\sum_{k=1}^{p}\sdepth(K[V({\mathcal{S}_{q}})]/I(\mathcal{S}_{q}))+\sdepth(K[\{\mathit{x}_{3}\}\cup\mathit{A}_{1}\cup\mathit{A}_{2}\cup\mathit{C}_{2}\cup\mathit{C}_{3}]).$$
  Using Theorem \ref{star}(a), $\depth(\mathit{Q}_{3,p,q}/(\mathit{J}_{3,p,q}:\mathit{x}_{3}))=\sdepth(\mathit{Q}_{3,p,q}/(\mathit{J}_{3,p,q}:\mathit{x}_{3}))=\sum_{k=1}^{p}1+1+2q+2pq=p+1+2q+2pq=(2q+1)(p+1).$ Now by Eq \ref{es8*}
  \begin{equation}\label{ff}
\mathit{Q}_{3,p,q}/(\mathit{J}_{3,p,q},\mathit{x}_{3})\cong \mathit{Q}^\star_{1,p,q}/\mathit{J}^\star_{1,p,q}\bigotimes_KK[\mathit{A}_{3}], \end{equation} by Lemma \ref{ref6}(a), $$\depth(\mathit{Q}_{3,p,q}/(\mathit{J}_{3,p,q},\mathit{x}_{3}))=\depth(\mathit{Q}^\star_{1,p,q}/\mathit{J}^\star_{1,p,q})+\depth(K[\mathit{A}_{3}])$$ and $$\sdepth(\mathit{Q}_{3,p,q}/(\mathit{J}_{3,p,q},\mathit{x}_{3}))=\sdepth(\mathit{Q}^\star_{1,p,q}/\mathit{J}^\star_{1,p,q})+\sdepth(K[\mathit{A}_{3}]),$$ by Eq \ref{ee}, $\depth(\mathit{Q}_{3,p,q}/(\mathit{J}_{3,p,q},\mathit{x}_{3}))=\sdepth(\mathit{Q}_{3,p,q}/(\mathit{J}_{3,p,q},\mathit{x}_{3}))=3p+2q+q=3(p+q).$ Thus by Eq \ref{g} and Eq \ref{ssg}, $\depth(\mathit{Q}_{3,p,q}/\mathit{J}_{3,p,q}),\sdepth(\mathit{Q}_{3,p,q}/\mathit{J}_{3,p,q})\geq 3(p+q).$ Now since $\mathit{x}_{31}\mathit{x}_{32}\dots\mathit{x}_{3q}\notin\mathit{J}_{3,p,q},$ so by Eq \ref{cc} and Lemma \ref{ref6}(a)
$$\depth(\mathit{Q}_{3,p,q}/(\mathit{J}_{3,p,q}:\mathit{x}_{31}\mathit{x}_{32}\dots\mathit{x}_{3q}))=\depth( \mathit{Q}^\star_{1,p,q}/\mathit{J}^\star_{1,p,q})+\depth(K[\mathit{A}_{3}]),$$ and
$$\sdepth(\mathit{Q}_{3,p,q}/(\mathit{J}_{3,p,q}:\mathit{x}_{31}\mathit{x}_{32}\dots\mathit{x}_{3q}))=\sdepth( \mathit{Q}^\star_{1,p,q}/\mathit{J}^\star_{1,p,q})+\sdepth(K[\mathit{A}_{3}]),$$
by Eq \ref{ee}, we have $\depth(\mathit{Q}_{3,p,q}/(\mathit{J}_{3,p,q}:\mathit{x}_{31}\mathit{x}_{32}\dots\mathit{x}_{3q}))=\sdepth(\mathit{Q}_{3,p,q}/(\mathit{J}_{3,p,q}:\mathit{x}_{31}\mathit{x}_{32}\dots\mathit{x}_{3q}))=3p+2q+q=3(p+q).$ Thus by Corollary \ref{c1} and Proposition \ref{c2}, $$\depth(\mathit{Q}_{3,p,q}/\mathit{J}_{3,p,q}),\sdepth(\mathit{Q}_{3,p,q}/\mathit{J}_{3,p,q})\leq 3(p+q).$$ Hence $\depth(\mathit{Q}_{3,p,q}/\mathit{J}_{3,p,q})=\sdepth(\mathit{Q}_{3,p,q}/\mathit{J}_{3,p,q})=3(p+q).$ If $n=4,$ then by using the similar arguments one can easily verify the required result that is $$\depth(\mathit{Q}_{4,p,q}/\mathit{J}_{4,p,q})=\sdepth(\mathit{Q}_{4,p,q}/\mathit{J}_{4,p,q})=4(p+q).$$ 
  Now let $n\geq5$. Applying Lemma \ref{ref6}(a) on Eq \ref{es7*}, we have $$\depth(\mathit{Q}_{n,p,q}/(\mathit{J}_{n,p,q}:\mathit{x}_{n}))=\depth(\mathit{Q}^\star_{n-4,p,q}/\mathit{J}^\star_{n-4,p,q})+\depth(K[\{\mathit{x}_{n}\}\cup\mathit{A}_{1}\cup\mathit{A}_{n-1}\cup\mathit{C}_{n-1}\cup\mathit{C}_{n}]),$$  
  and  $$\sdepth(\mathit{Q}_{n,p,q}/(\mathit{J}_{n,p,q}:\mathit{x}_{n}))=\sdepth(\mathit{Q}^\star_{n-4,p,q}/\mathit{J}^\star_{n-4,p,q})+\sdepth(K[\{\mathit{x}_{n}\}\cup\mathit{A}_{1}\cup\mathit{A}_{n-1}\cup\mathit{C}_{n-1}\cup\mathit{C}_{n}]),$$
using Lemma \ref{Lem4} $$\depth(\mathit{Q}_{n,p,q}/(\mathit{J}_{n,p,q}:\mathit{x}_{n}))=(p+q)(n-4+1)+p+2pq+2q+1=(p+q)(n-3)+2pq+2q+p+1,$$
and
$$\sdepth(\mathit{Q}_{n,p,q}/(\mathit{J}_{n,p,q}:\mathit{x}_{n}))=(p+q)(n-4+1)+p+2pq+2q+1=(p+q)(n-3)+2pq+2q+p+1.$$
Now applying Lemma \ref{ref6}(a) on Eq \ref{es8*}, we have  $$\depth(\mathit{Q}_{n,p,q}/(\mathit{J}_{n,p,q},\mathit{x}_{n}))=\depth(\mathit{Q}^\star_{n-2,p,q}/\mathit{J}^\star_{n-2,p,q})+\depth(K[\mathit{A}_{n}]),$$
and
$$\sdepth(\mathit{Q}_{n,p,q}/(\mathit{J}_{n,p,q},\mathit{x}_{n}))=\sdepth(\mathit{Q}^\star_{n-2,p,q}/\mathit{J}^\star_{n-2,p,q})+\sdepth(K[\mathit{A}_{n}]),$$
		 again using Lemma \ref{Lem4}, $\depth(\mathit{Q}_{n,p,q}/(\mathit{J}_{n,p,q},\mathit{x}_{n}))=\sdepth(\mathit{Q}_{n,p,q}/(\mathit{J}_{n,p,q},\mathit{x}_{n}))=(p+q)(n-2+1)+p+q=(p+q)n.$ Hence  it follows by Eq \ref{g} and Eq \ref{ssg} that, $$\depth(\mathit{Q}_{n,p,q}/\mathit{J}_{n,p,q}),\sdepth(\mathit{Q}_{n,p,q}/\mathit{J}_{n,p,q})\geq (p+q)n.$$ Now since $\mathit{x}_{n1}\mathit{x}_{n2}\dots\mathit{x}_{nq}\notin\mathit{J}_{n,p,q},$ so by Eq \ref{cc} and Lemma \ref{ref6}(a)         $$\depth(\mathit{Q}_{n,p,q}/(\mathit{J}_{n,p,q}:\mathit{x}_{n1}\mathit{x}_{n2}\dots\mathit{x}_{nq}))=\depth( \mathit{Q}^\star_{n-2,p,q}/\mathit{J}^\star_{n-2,p,q})+\depth(K[\mathit{A}_{n}]),$$  and        $$\sdepth(\mathit{Q}_{n,p,q}/(\mathit{J}_{n,p,q}:\mathit{x}_{n1}\mathit{x}_{n2}\dots\mathit{x}_{nq}))=\sdepth( \mathit{Q}^\star_{n-2,p,q}/\mathit{J}^\star_{n-2,p,q})+\sdepth(K[\mathit{A}_{n}]).$$ 
         So by Lemma \ref{Lem4}, $$\depth(\mathit{Q}_{n,p,q}/(\mathit{J}_{n,p,q}:\mathit{x}_{n1}\mathit{x}_{n2}\dots\mathit{x}_{nq}))=\sdepth(\mathit{Q}_{n,p,q}/(\mathit{J}_{n,p,q}:\mathit{x}_{n1}\mathit{x}_{n2}\dots\mathit{x}_{nq}))=(p+q)n.$$
		Again using Corollary \ref{c1} and Proposition \ref{c2}, $$\depth(\mathit{Q}_{n,p,q}/\mathit{J}_{n,p,q}),\sdepth(\mathit{Q}_{n,p,q}/\mathit{J}_{n,p,q})\leq (p+q)n.$$ Hence $\depth(\mathit{Q}_{n,p,q}/\mathit{J}_{n,p,q})=\sdepth(\mathit{Q}_{n,p,q}/\mathit{J}_{n,p,q})=(p+q)n.$
        
 If $n=3$, then by Eq \ref{dd} and Lemma \ref{ref6}(b) $$\reg(\mathit{Q}_{3,p,q}/(\mathit{J}_{3,p,q}:\mathit{x}_{3}))=\sum_{k=1}^{p}\reg(K[V({\mathcal{S}_{q}})]/I(\mathcal{S}_{q})),$$ by Theorem \ref{star}(b), $\reg(\mathit{Q}_{3,p,q}/(\mathit{J}_{3,p,q}:\mathit{x}_{3}))=\sum_{k=1}^{p}1=p.$ Similarly, by Eq \ref{ff} and Lemma \ref{ref6}(b) $\reg(\mathit{Q}_{3,p,q}/(\mathit{J}_{3,p,q},\mathit{x}_{3}))= \reg(\mathit{Q}^\star_{1,p,q}/\mathit{J}^\star_{1,p,q})$ and by Eq \ref{hh}, $$\reg(\mathit{Q}_{3,p,q}/(\mathit{J}_{3,p,q},\mathit{x}_{3}))=3p.$$ The required result follows by Theorem \ref{reg}(c), that is, $\reg(\mathit{Q}_{3,p,q}/\mathit{J}_{3,p,q})=3p.$ If $n=4$, the result can be proved by using similar arguments that is $\reg(\mathit{Q}_{4,p,q}/\mathit{J}_{4,p,q})=4p.$
 
 Now let $n\geq5$. Using Eq \ref{es7*} and Lemma \ref{ref6}(b), we have
$$\reg(\mathit{Q}_{n,p,q}/(\mathit{J}_{n,p,q}:\mathit{x}_{n}))=\reg(\mathit{Q}^\star_{n-4,p,q}/\mathit{J}^\star_{n-4,p,q}).$$ Similarly, by Eq \ref{es8*} and Lemma \ref{ref6}(b) $\reg(\mathit{Q}_{n,p,q}/(\mathit{J}_{n,p,q},\mathit{x}_{n}))=\reg(\mathit{Q}^\star_{n-2,p,q}/\mathit{J}^\star_{n-2,p,q}).$ By Lemma \ref{Lem4}, $\reg(\mathit{Q}_{n,p,q}/(\mathit{J}_{n,p,q}:\mathit{x}_{n}))=(n-4+2)p=(n-2)p$ and  $\reg(\mathit{Q}_{n,p,q}/(\mathit{J}_{n,p,q},\mathit{x}_{n}))=(n-2+2)p=np.$ Hence by Theorem \ref{reg}(c),
			$	\reg(\mathit{Q}_{n,p,q}/\mathit{J}_{n,p,q})=np.$ 
   
    Now by Auslander–Buchsbaum formula, we have $${\pdim}(Q_{n,p,q}/\mathit{J}_{n,p,q})={\depth}(Q_{n,p,q})-{\depth}(Q_{n,p,q}/\mathit{J}_{n,p,q}).$$ Hence ${\pdim}(Q_{n,p,q}/\mathit{J}_{n,p,q})=n(q+1)(p+1)-n(p+q)=n(1+pq).$
	\end{proof}
	
\end{document}